\documentclass{amsart}

\usepackage[english]{babel}
\usepackage[utf8]{inputenc}

\usepackage{graphicx}
\usepackage{color}
\usepackage{caption}
\usepackage{subcaption}

\usepackage{mathtools}
\usepackage{xcolor}

\usepackage{amsmath}
\usepackage{amssymb}
\usepackage{algorithm}

\usepackage{enumitem}
\usepackage{diagbox}

\usepackage{mathabx}
\usepackage{tikz}
\usetikzlibrary{graphs}
\usetikzlibrary{calc}

\usepackage{tkz-graph}
\usetikzlibrary{fit}

\usepackage{hyperref}
\usepackage{multido}

\usepackage{algpseudocode}

\DeclarePairedDelimiter\floor{\lfloor}{\rfloor}

\numberwithin{equation}{section}

\allowdisplaybreaks

\DeclarePairedDelimiter\abs{\lvert}{\rvert}%
\newtheorem{theorem}{Theorem}[section]
\newtheorem{lemma}[theorem]{Lemma}

\newtheorem{assumption}{Assumption}

\theoremstyle{definition}
\newtheorem{definition}[theorem]{Definition}

\calclayout

\DeclareMathOperator*{\argmin}{argmin \,}

\title{Multi-resolution low-rank tensor formats}
 
 \author{Oscar Mickelin}
\address{Department of Mathematics, Massachusetts Institute of Technology, Massachusetts, USA}
\email{oscarmi@mit.edu}

\author{Sertac Karaman}
\address{Department of Aeronautics and Astronautics, Massachusetts Institute of Technology, Massachusetts, USA}
\email{sertac@mit.edu}
\thanks{This work was supported in part by the Army Research Office grant no. W911NF1510249 and the ARL DCIST program.}



\begin{document}
\maketitle

\begin{abstract}
We describe a simple, black-box compression format for tensors with a multiscale structure. By representing the tensor as a sum of compressed tensors defined on increasingly coarse grids, we capture low-rank structures on each grid-scale, and we show how this leads to an increase in compression for a fixed accuracy. We devise an alternating algorithm to represent a given tensor in the multiresolution format and prove local convergence guarantees. In two dimensions, we provide examples that show that this approach can beat the Eckart-Young theorem, and for dimensions higher than two, we achieve higher compression than the tensor-train format on six real-world datasets. We also provide results on the closedness and stability of the tensor format and discuss how to perform common linear algebra operations on the level of the compressed tensors.
\end{abstract}

\section{Introduction}
High-dimensional data is often represented with tensors in $\mathbb{R}^{n_1 \times \ldots \times n_d}$. When the dimensions $n_k$ are large, compressed tensor formats are often used to cope with the large storage requirements and operational costs. The tensor-train format \cite{oseledets2011tensor} and the canonical decomposition \cite{kolda2009tensor} are two formats that have received significant attention over the last decade. These formats represent tensors using different black-box low-rank expansions, and achieve significant reduction in storage costs when the associated rank is low.

In this article, we will be interested in tensors that are not necessarily low-rank in either of these formats, namely tensors with multiple length-scales. Even when the contribution to the tensor from each length-scale has low rank, their combination might be of significantly higher rank. Multiscale data is essential in numerous scientific and engineering problems \cite{deMoraes2019,lee2017stochastic,lee2015multiscale,zou2005multiscale}. In many cases, the memory limitations, particularly on emerging edge computing devices, require that large-scale multiscale data is compressed for storage and for processing in compressed formats. 

There have been a few approaches to compressing multiscale tensors in the literature. In one approach \cite{ozdemir2017multiscale,ozdemir2016multiscale,ozdemir2017multi}, the authors compress a given tensor into a tree structure of compressed subtensors, by recursively subdividing a tensor into local blocks on different scales and decomposing each block on each scale in the tree in the Tucker format using the HOSVD \cite{de2000multilinear}. A similar approach has been pursued for matrices \cite{ong2016beyond}, where a convex nuclear-norm relaxation is used to recover the local low-rank structures. Wu et al. \cite{WXY:07b,wu2008hierarchical} consider tensors representing visual data, which are partitioned into blocks on increasingly finer scales. All blocks on each scale are represented with a common basis in the Tucker format, to capture global correlations of a locally repeating structure. For fast, GPU-accelerated and interactive visualization of data, previous work \cite{SIMAEZGGP:11,SMP:13,BSP:18} has also considered subsampling a given tensor decomposition at runtime, into a desired visualization resolution. The data is again represented by dividing the tensor into local blocks, which are then compressed. Khoromskij and Khoromskaia \cite{khoromskij2009multigrid} present different multigrid-inspired techniques with the goal of improving the speed of decomposition into the Tucker format, rather than decreasing the storage size of the resulting decomposition. A subsampling onto coarser grids identifies indices carrying the most information, after which alternating least squares algorithms can be run on these subindices. A related, but distinct approach is the quantized tensor train-approximation; we refer to e.g., a recent monograph \cite{khoromskij2018tensor} for more information.

We will consider a simple approach to improving the compression ratios of tensors that exhibit a multiscale structure. This approach complements the existing literature, in that our approach captures multiscale tensors where each scale gives a non-local contribution to the tensor. Starting from a tensor given in a compressed format, we introduce a graded structure on the domain and capture low-rank information on different scales. These scales have spatial resolution with increasing coarseness, which leads to a corresponding decrease in the representation cost. This multiresolution format enables us to achieve higher compression while maintaining a given accuracy. In the case of matrices, we show how this makes it possible to beat the Eckart-Young theorem \cite{golub2013matrix}, by achieving lower approximation error than the truncated singular value decomposition, for a given storage cost. We devise an alternating algorithm for decomposing a tensor into the multiresolution format, where the tensor can be provided in either full format or the underlying, non-multiresolution tensor format. We prove a local convergence result of a slightly restructured version of this algorithm. We also discuss the closedness of the multiresolution format and how to perform common linear algebra operations in the format in ways that respects the graded structure.

The remainder of the article is structured as follows. Section~\ref{sec:notation} explains our notation and introduces operators that convert tensors between differently coarse grids. Section~\ref{sec:format} motivates and introduces the multiresolution format, and Section~\ref{sec:closedness} contains results on closedness of the multiresolution format. Section~\ref{sec:alg} describes an alternating algorithm for computing a decomposition in the multiresolution format and Appendix~\ref{sec:local} proves local convergence of a slightly restructured algorithm. In Section~\ref{sec:operations}, we discuss how to perform common tensor operations on the level of the compressed tensors, and we conclude by studying the performance of the multiresolution format for several examples of real-world tensors in Section~\ref{sec:applications}. Implementations of all algorithms in this paper are publicly available online. \footnote{\href{https://github.com/MultiResTF/multiresolution}{\texttt{https://github.com/MultiResTF/multiresolution}}}

\section{Notation}\label{sec:notation}
Throughout the article, we will refer to a number of low-rank tensor formats, for instance the tensor-train format, the hierarchical format \cite{hackbusch2012tensor}, the canonical decomposition \cite{kolda2009tensor}, or orthogonally constrained canonical decompositions \cite{chen2009tensor,anandkumar2014tensor}. We will denote the set of tensors represented in a general format by $\mathcal{F}$. Each of these comes with a corresponding notion of rank, which we will denote by $\text{rank}_{\mathcal{F}}$. We denote by $\mathcal{F}_r$ the set of tensors in $\mathcal{F}$ with corresponding rank no greater than $r$. For the canonical decomposition, $r$ is a positive integer, and for the tensor-train format, $r$ is a vector of positive integers and inequalities between these vectors are interpreted element-wise. We will also consider low-rank matrices, for which the tensor-train format and canonical format coincide with the ordinary low-rank matrix format. For tensor formats $\mathcal{F}$ that are weakly closed, there is an optimal approximation $T_{\text{opt}}$ in $\mathcal{F}_{\mathbf{r}}$ to any tensor $T$ \cite[Thm.~4.28]{hackbusch2012tensor}, and we will denote this by $\text{round}_{\mathcal{F}_\mathbf{r}}(T)$. For the tensor-train format, the TT-SVD procedure \cite[Alg.~1]{oseledets2011tensor} efficiently produces a quasi-optimal tensor $\widetilde{T}$ in $\mathcal{F}_{\mathbf{r}}$, i.e., satisfying $\|T - \widetilde{T}\| \leq \sqrt{d-1}\|T - T_{\text{opt}}\|$. $T$ can be given either in full format or in the TT-format. For certain tensor formats, e.g., the tensor-train format, it is also possible to instead specify an error bound $\varepsilon$. A rounding procedure then produces an approximation $\widetilde{T}$ with lower rank than $T$, guaranteed to satisfy $\|T - \widetilde{T}\| \leq \varepsilon \|T\|$. We will denote this procedure by $\text{round}_\mathcal{F}(T, \varepsilon)$. We will denote the inner product between two tensors $T$ and $S$ in $\mathbb{R}^{n_1 \times \ldots \times n_d}$ by
\begin{equation}
\langle T, S \rangle := \sum_{i_1 =1}^{n_1} \ldots \sum_{i_d=1}^{n_d} T(i_1, \ldots , i_d)S(i_1, \ldots , i_d),
\end{equation}
and the Frobenius norm of $T$ is defined to be $\|T\| := \sqrt{\langle T, T\rangle}$.

Throughout this paper, we will fix a batch size $b_s \in \mathbb{N}$ with $b_s \geq 2$. We will decompose tensors into a sum of tensors defined on grids with increasingly coarse resolution. The $k$:th coarsest level will consist of tensors constant on blocks with side length $b_s^k$. The parameter $b_s$ therefore controls the resolution of the subsequent grids, and therefore also defines a grid-refinement scheme defined in the next paragraph. We will in the following consider $L$ levels, with the coarsest one having blocks of size $b_s^L$. This therefore requires that $b_s^L$ divides $n_1, n_2, \ldots , n_d$. In order to describe this construction in detail, we introduce the following two operations, which will be heavily used in what follows.

Given a tensor $T \in \mathbb{R}^{ b_s^k \times \ldots \times b_s^k}$ and a positive integer $\ell$, we define the block-extended tensor $\text{ext}_\ell(T) \in \mathbb{R}^{ b_s^{k+\ell} \times \ldots \times b_s^{k+\ell}}$ by
\begin{equation}
\text{ext}_\ell(T)(i_1, \ldots , i_d) = T\left(\floor*{\frac{i_1-1}{b_s^\ell}}+1, \ldots, \floor*{\frac{i_d-1}{b_s^\ell}}+1\right),
\end{equation}
i.e., the tensor obtained by replacing each entry of $T$ by a block with side lengths $b_s^\ell$, where each entry equals the replaced entry of $T$. As an example, if $T$ is a $2\times2$-tensor, and $S$ is a $2\times 2 \times 2$-tensor defined by
\begin{equation}
T = \left[\begin{array}{@{}cc@{}}
a & b  \\ c & d \end{array}\right], \qquad S = \left[\begin{array}{@{}cc|cc@{}}
a & b & e & f \\ c & d & g & h \end{array}\right],
\end{equation}
and if $b_s = 2$, then their first extended tensors are given by
\begin{align}
\arraycolsep=2.45pt  \text{ext}_1(T) = \begin{bmatrix}
a & a & b & b \\
a & a & b & b \\ 
c & c & d & d \\
c & c & d & d
\end{bmatrix}\!\!, \,\,\,\, \arraycolsep=2.45pt \text{ext}_1(S) = \left[\begin{array}{@{}cccc|cccc|cccc|cccc@{}}
a &a & b &b & a &a & b &b & e &e & f &f & e &e & f & f \\
a &a & b &b & a &a & b &b & e &e & f &f & e &e & f & f \\
c &c & d &d &c & c & d &d & g &g & h&h& g &g & h& h \\
c &c & d &d &c & c & d &d & g &g & h&h& g &g & h& h
 \end{array}\right],
\end{align}
respectively. Similarly, we denote the left inverse of $\text{ext}_\ell$ by $\text{ave}_\ell$. For a tensor $S \in \mathbb{R}^{ b_s^{k+\ell} \times \ldots \times b_s^{k+\ell}}$, $\text{ave}_\ell(S) \in \mathbb{R}^{b_s^{k} \times \ldots \times b_s^{k}}$ is defined by
\begin{equation}
\text{ave}_\ell (S)(i_1, \ldots , i_d)  = \frac{1}{b_s^{\ell d}}\sum_{j_1 = 0}^{b_s^{\ell} - 1} \ldots \sum_{j_d = 0}^{b_s^{\ell} - 1} S(i_1 + j_1, \ldots , i_d + j_d).
\end{equation}
Clearly
\begin{equation}\label{eq:extprop}
\begin{split}
\text{ave}_{\ell_1 + \ell_2}(T) = \text{ave}_{\ell_1}  &\left( \text{ave}_{\ell_2}(T) \right), \quad  \text{ext}_{\ell_1 + \ell_2}(T) = \text{ext}_{\ell_1}  \left( \text{ext}_{\ell_2}(T) \right), \\
\text{ave}_\ell(\text{ext}_\ell(T)) &= T, \quad \text{ave}_\ell(\text{ext}_k(T)) = \text{ext}_{k-\ell}(T),
\end{split}
\end{equation}
for any integers $\ell$, $\ell_1$, $\ell_2$ and $k$, with $ \ell \leq k$. These operations will allow us to convert tensors into finer or coarser grids. 

\section{Tensor format}\label{sec:format}
Let $T$ be a tensor in $\mathbb{R}^{ b_s^L \times \ldots \times b_s^L}$, for some positive integer $L$. We will approximate $T$ by a sum of subtensors defined on grids with increasing coarseness. Each subtensor will be represented in a compressed tensor format denoted by $\mathcal{F}$. We make the following definition.
\begin{definition}
Let $\mathbf{r} = (r_0, \ldots, r_L)$ be a vector of rank bounds for each grid-scale. For any compressed tensor format $\mathcal{F}$, we define the multiresolution $\mathcal{F}_\mathbf{r}$-format by
\begin{equation}
\text{MS}_{\mathcal{F}_\mathbf{r}}= \left\{ T: T = \sum_{k=0}^L \text{ext}_{L-k}(T_k), T_k \in \mathcal{F}_{r_k}, T_k \in \mathbb{R}^{b_s^k \times \ldots \times b_s^k}\right\}.
\end{equation}
To represent a tensor in $MS_{\mathcal{F}_\mathbf{r}}$, we only need to store the $L+1$ tensors $T_k$, for $0 \leq k \leq L$. We will say that $T$ has the multiresolution representation $(T_0, \ldots , T_L)$.
\end{definition}
 The motivation behind the definition is that storing as much information as possible on coarser scales decreases the total storage cost of the tensor, since fewer grid points need to be kept in memory as compared to the finest scale. The multiresolution format uses the operator $\text{ext}_{L-k}$, instead of a more smooth interpolation operator as is common in multigrid methods, since we will see in Section~\ref{sec:alg} that this allows for a simple algorithm to approximate a given tensor in $\mathcal{F}_{\mathbf{r}}$.
 
Since the format $ \text{MS}_{\mathcal{F}_\mathbf{r}}$ contains rank-$r_L$ approximations on the finest scale, $\mathcal{F}_{r_L} \subseteq \text{MS}_{\mathcal{F}_\mathbf{r}}$, and the multiresolution format contains any tensor for $r_L$ large enough. The contribution $\text{ext}_{L-k}(T_k)$ on each scale is contained in $\mathcal{F}_{r_k}$. When $\mathcal{F}$ is either the tensor train format or the canonical decomposition, it is therefore also the case that $\text{MS}_{\mathcal{F}_\mathbf{r}} \subseteq \mathcal{F}_{r_0 + \ldots + r_L}$. However, the storage cost of a tensor $T$ in $\text{MS}_{\mathcal{F}_\mathbf{r}}$ is lower than that of a tensor in $\mathcal{F}_{r_0 + \ldots + r_L}$ since the corresponding tensors $T_k$ are compressed versions of tensors on the smaller index set $\mathbb{R}^{b_s^k \times \ldots \times b_s^k}$, instead of on the the larger index set $\mathbb{R}^{b_s^L \times \ldots \times b_s^L}$.

For a wide range of accuracies, our examples in Section~\ref{sec:applications} will show that approximations in the multiresolution format can often require lower storage costs than in $\mathcal{F}$. Note however that we do not expect any storage gains when representing a tensor to machine precision in the format $\text{MS}_{\mathcal{F}_\mathbf{r}}$, as compared to storing the tensor in the format $\mathcal{F}_{r_L}$. The following example explains why this is the case.
\subsection{Motivating example}\label{sec:motivating_example}
We consider a function with multiple length-scales, for instance
\begin{equation}\label{eq:motivating_example}
f(x) =  \prod_{k=1}^d\sin\left(x_k\right) +  \prod_{k=1}^d\sin\left(2x_k\right)+  \prod_{k=1}^d\sin\left(4x_k\right).
\end{equation}
We let $T$ be the grid-discretization of $f$ on the interval $[0,\pi]$ using a uniform grid with a total of $n$ grid points in each dimension. $T$ therefore has a canonical representation
\begin{equation}\label{eq:ONex}
T = \bigotimes_{k=1}^d u_k +  \bigotimes_{k=1}^d v_k + \bigotimes_{k=1}^d w_k,
\end{equation}
where $u_k$ is a discretization of $\sin(x_k)$ on the interval $[0,\pi]$, $v_k$ of $\sin(2x_k)$ and $w_k$ of $\sin(4x_k)$. We then have $\langle u_k, v_k\rangle = 0 = \langle u_k , w_k \rangle = \langle v_k , w_k \rangle$ when $n$ is odd. Eq.~\eqref{eq:ONex} therefore describes an orthogonal canonical decomposition \cite{kolda2001orthogonal}. An optimal rank $2$-approximation of $T$ is then obtained by keeping the two terms in Eq.~\eqref{eq:ONex} with the largest norms \cite{zhang2001rank}. In our case, the terms have equal norms so we retain any two terms, e.g., $\bigotimes_{k=1}^d v_k +  \bigotimes_{k=1}^d w_k$. The square of the approximation error is
\begin{equation}
\|\bigotimes_{k=1}^d u_k\|^2 = \left(\frac{n}{\pi} \int_0^\pi \sin^2(x)  \text{d}x  +\mathcal{O}(1) \right)^d= \left( \frac{n}{2} \right)^d + \mathcal{O}(n^{d-1}).
\end{equation}
A possible (but not necessarily optimal) multiresolution approximation with $b_s = 2$ would be $\widetilde{T} = \widetilde{T}_1 + \text{ext}_1(\widetilde{T}_2) + \text{ext}_2(\widetilde{T}_3)$, where $\widetilde{T}_1 = \bigotimes_{k=1}^d w_k$, $\widetilde{T}_2 = \bigotimes_{k=1}^d \text{ave}_1(v_k)$, and $\widetilde{T}_3 = \bigotimes_{k=1}^d \text{ave}_2(u_k)$. Since $\text{ave}_1(v_k) \in \mathbb{R}^{\frac{n}{2}}$ and $\text{ave}_2(u_k) \in \mathbb{R}^{\frac{n}{4}}$, the cost of storing $\widetilde{T}_1$, $\widetilde{T}_2$, and $\widetilde{T}_3$ is less than storing the optimal rank-2 approximation. By the following result, this also results in far lower approximation error.

\begin{theorem}
Let $\omega_1, \omega_2, \ldots , \omega_r$ be an increasing sequence of positive real numbers and $f$ a function with multiple length-scales, written in the form
\begin{equation}
f(x) = \sum_{k=1}^r \prod_{j=1}^d g_{kj}\left( \frac{x_j}{\omega_k}\right),
\end{equation}
where $\| g'_{kj}(x_j) \|_{\infty} \leq C_{kj}$ for all $k,j$. Let $T$ be the corresponding discretization on the hypercube $[a,b]^d$ with $n = b_s^L$ uniform grid points in each dimension, i.e., $T = \sum_{k=1}^r T_k$ with $T_k = \bigotimes_{j=1}^d u_{kj}$. The multiresolution canonical approximation $\widetilde{T} = \widetilde{T}_1 + \sum_{k=2}^r \text{ext}_{k-1}(\widetilde{T}_{k})$ with $\widetilde{T}_{k} = \bigotimes_{j=1}^d \text{ave}_{k-1}\left(u_{k,j}\right)$ then satisfies
\begin{equation}\label{eq:bound_mot}
\|T- \widetilde{T}\| \leq \sum_{k=2}^r \delta_k \|T_k\|,
\end{equation}
where $\delta_k = \left[ \prod_{j=1}^d \left(1 + (b-a)\frac{C_{k,j}\sqrt{ \frac{1}{60}\left(7b_s^{2(k-1)} - 15 + 8b_s^{-2(k-1)}\right)}} {\sqrt{n}\omega_{k}\|u_{k,j}\|} \right) \right] - 1.$ For large $n$, the right hand side of Eq.~\eqref{eq:bound_mot} is approximately equal to 
\begin{equation}
\sum_{k=2}^r\sum_{j=1}^d (b-a)\frac{C_{k,j} \sqrt{ \frac{1}{60}\left(7b_s^{2(k-1)} - 15 + 8b_s^{-2(k-1)}\right)} }{\sqrt{n}\omega_{k}\|u_{k,j}\|}  \|T_k\|.
\end{equation}
\end{theorem}
\begin{proof}
A sensitivity formula for the canonical decomposition \cite[Prop.~7.10]{hackbusch2012tensor} gives the error bound in Eq.~\eqref{eq:bound_mot} with
\begin{equation}
\delta_k = \left[ \prod_{j=1}^d \left(1 + \frac{\|u_{k,j} - \text{ext}_{k-1}\left(\text{ave}_{k-1}(u_{k,j})\right)\|}{\|u_{k,j}\|} \right) \right] - 1,
\end{equation}
and we only need to bound the quantity on the right. Fix now the indices $k$ and $j$. For any batch $B$ of indices of length $b_s^{k-1}$, the average of $u_{k,j}$ over $B$ is $\frac{1}{b_s^{k-1}}\sum_{m\in B}u_{k,j}(m)$. We have
\begin{equation}\label{eq:error_mot}
\begin{split}
\|u_{k,j} - \text{ext}_{k-1}\left(\text{ave}_{k-1}(u_{k,j})\right)\|^2_F
= \sum_{\text{B}} \sum_{i \in B} \left(u_{k,j}(i) - \frac{1}{b_s^{k-1}}\sum_{m\in B}u_{k,j}(m)  \right)^2 \\= \sum_{\text{B}} \sum_{i \in B} \frac{1}{b_s^{2(k-1)}}\left(\sum_{m\in B} \left[u_{k,j}(i) - u_{k,j}(m)\right]  \right)^2.
\end{split}
\end{equation}
Now
\begin{equation}
\begin{split}
\abs*{u_{k,j}(i) - u_{k,j}(m)} = \abs*{g_{k,j}\left( \frac{i(b-a)}{n\omega_{k}}\right) - g_{k,j}\left( \frac{m(b-a)}{n\omega_{k}}\right)} \\ \leq 
\frac{(b-a)\abs{i-m} C_{k,j}}{n\omega_{k}} .
\end{split}
\end{equation}
Inserting this into Eq.~\eqref{eq:error_mot} results in 
\begin{equation}\label{eq:almost_done_mot}
\begin{split}
\|u_{k,j} - \text{ext}_{k-1}\left(\text{ave}_{k-1}(u_{k,j})\right)\|^2_F 
\leq  \sum_{\text{B}} \sum_{i \in B} \frac{(b-a)^2C_{k,j}^2}{n^2\omega_{k}^2b_s^{2(k-1)}}\left(\sum_{m\in B} \abs{i-m}  \right)^2 \\
= b_s^{L-k+1}\frac{(b-a)^2C_{k,j}^2}{n^2\omega_{k}^2b_s^{2(k-1)}} \sum_{i =1}^{b_s^{k-1}}\left(\sum_{m=1}^{b_s^{k-1}} \abs{i-m}  \right)^2.
\end{split}
\end{equation}
Using elementary closed-form expressions, the sum in the right hand side can be evaluated to be
\begin{equation}
\sum_{i =1}^{b_s^{k-1}}\left(\sum_{m=1}^{b_s^{k-1}} \abs{i-m}  \right)^2 = \frac{b_s^{k-1}}{60} \left( 7b_s^{4(k-1)} - 15b_s^{2(k-1)} + 8\right),
\end{equation}
so inserting this into Eq.~\eqref{eq:almost_done_mot} and using the fact that $n = b_s^L$, we obtain
\begin{equation}
\begin{split}
\|u_{k,j} \! - \text{ext}_{k-1}\!\left(\text{ave}_{k-1}(u_{k,j})\right)\|^2_F
\leq \!\frac{(b-a)^2C_{kj}^2}{n\omega_k^2} \frac{1}{60}\left( 7b_s^{2(k-1)} \! - 15 + 8b_s^{-2(k-1)}\!\right),
\end{split}
\end{equation}
which concludes the proof.
\end{proof}

In the motivating example, we can take $C_{kj} = 1$, $\omega_1 = \frac{1}{4}$, $\omega_2 = \frac{1}{2}$ and $\omega_3 = 1$ to conclude
\begin{equation}
\|T- \widetilde{T}\| \leq \sum_{k=2}^r \delta_k \|T_k\| \approx \frac{9.65 \pi }{n} \|\bigotimes_{k=1}^d u_k\|,
\end{equation}
which shows that the multiresolution format can achieve far lower approximation error for a given storage cost, provided $n$ is large enough. Our computational examples in Section~\ref{sec:applications} demonstrate that gains in storage can be achieved also for moderate values of $n$.

However, this example also shows that the error when representing a tensor in the multiresolution format has an inherent lower bound from using a coarser grid. To achieve machine precision for a general tensor, we would therefore in general expect to need to use the rank vector $(0, \ldots , 0, r_L)$ with $r_L$ large enough. However, for lower accuracy, the example above shows that it is possible to obtain good storage gains.

\section{Closedness and stability}\label{sec:closedness}
When attempting to find an optimal approximation of a tensor $T$ in $\text{MS}_{\mathcal{F}_{\mathbf{r}}}$ for a fixed rank vector $\mathbf{r}$, it is important to know whether or not the set $\text{MS}_{\mathcal{F}_{\mathbf{r}}}$ is closed. If not, then a tensor $T \in \overline{\text{MS}_{\mathcal{F}_{\mathbf{r}}}} \setminus \text{MS}_{\mathcal{F}_{\mathbf{r}}}$ by definition has a corresponding sequence of tensors $T^{(n)}$ in $ \text{MS}_{\mathcal{F}_{\mathbf{r}}}$, converging to $T$. $T$ therefore does not have an optimal approximation in the set $\text{MS}_{\mathcal{F}_{\mathbf{r}}}$ so the problem is ill-posed. In the by now classical setting of the (non-multiresolution) canonical format, this is associated with an instability in that successive approximations $T^{(n)}$ to $T$ have terms with diverging norm and convergence to $T$ is achieved through unstable cancellation effects \cite{de2008tensor}. We now show that the same holds true for the multiresolution format, even when using a closed format on each scale.

The base example is the multiresolution low-rank matrix format for $d=2$ with $b_s = 2$ and rank vector $(1,1)$. The matrix
\begin{equation}
T^{(n)} = \begin{bmatrix}
n & n \\
n & n
\end{bmatrix} - 
\begin{bmatrix}
\sqrt{n+1} \\
\sqrt{n-1}
\end{bmatrix}
\begin{bmatrix}
\sqrt{n+1} & 
\sqrt{n-1}
\end{bmatrix}
\end{equation}
is contained in the multiresolution format with rank vector $(1,1)$ for any $n$, and $T^{(n)} \rightarrow T := \bigl[ \begin{smallmatrix} -1 & 0 \\ 0 & 1\end{smallmatrix}\bigr]$, since $n - \sqrt{n+1}\sqrt{n-1} = \frac{n^2 - (n^2-1)}{n + \sqrt{n^2-1}} \rightarrow 0$ as $n\rightarrow \infty$. Since $T + a \bigl[ \begin{smallmatrix} 1 & 1 \\ 1 & 1\end{smallmatrix}\bigr] = \bigl[ \begin{smallmatrix} a-1 & a \\ a & a+1\end{smallmatrix}\bigr]$ can be seen to have rank $2$ for any real number $a$ by row reduction, it follows that $T$ is not in the multiresolution format with rank vector $(1,1)$. In other words, the format is not closed. We next extend this example to general rank vectors, values of $b_s$ and higher dimension $d$. We will consider the tensor-train format as the base format. Even though the tensor-train format is closed, we will show that the resulting multiresolution format $MS_{\text{TT}_\mathbf{r}}$ is not closed, in general. Here, the multiresolution rank vector is $\mathbf{r} = (r_0, \ldots , r_L)$, where each $r_k$ is a vector of tensor-train ranks, i.e., $r_k = ((r_{k})_1, \ldots , (r_{k})_{d-1})$.
\begin{theorem}\label{thm:closed}
~\begin{enumerate}
\item For $d=2$, the format $MS_{\text{TT}_\mathbf{r}}$ is closed if and only if $\mathbf{r}$ is of either the form $(r_0,\ldots, r_{k-1}, b_s^{k}, 0, \ldots , 0)$ or $(0, 0, \ldots, 0, r_k, 0, \ldots , 0)$.
\item For $d\geq 3$, the format $MS_{TT_\mathbf{r}}$ is not closed if $\mathbf{r} = (r_0,\ldots, r_{k-1}, r_k, 0, \ldots , 0)$, where the first tensor rank of $r_k$, $(r_k)_1$, is strictly less than $b_s^k$ and not all $(r_i)_1$ are zero, for $i < k$.
\end{enumerate}
\end{theorem}

The proof of Thm.~\ref{thm:closed} is detailed in Appendix~\ref{appendix:proof_closed}. Thm.~\ref{thm:closed} shows that there is no reason to expect $MS_{\mathcal{F}_{\mathbf{r}}}$ to be closed, even when the underlying format $\mathcal{F}$ is closed. However, we now prove a stability property that is only achieved when using a stable tensor format on each grid-scale. The following definition is similar to one made for the non-multiresolution canonical decomposition \cite[Def.~9.15]{hackbusch2012tensor}.
\begin{definition}
A sequence of tensors $T^{(n)}$ in $MS_{\mathcal{F}_{\mathbf{r}}}$ with 
\begin{equation}
T^{(n)} = \sum_{k=0}^{L} \text{ext}_{L-k}(T_k^{(n)})
\end{equation}
is called stable if there is a constant $C < \infty$ such that $\|T_k^{(n)}\| \leq C \|T^{(n)}\|$ for each $k = 0, \ldots, L$ and $n$.
\end{definition}
\begin{theorem}
The format $\mathcal{F}$ is closed if and only if, for all possible rank vectors $\mathbf{r}$, all stable, convergent sequences in $MS_{\mathcal{F}_{\mathbf{r}}}$ converge to a tensor in $MS_{\mathcal{F}_{\mathbf{r}}}$.
\end{theorem}
\begin{proof}
By taking $\mathbf{r} = (0, 0, \ldots, r)$, the ``only if'' part follows. For the converse, let $T^{(n)}$ be any sequence of stable tensors in $MS_{\mathcal{F}_{\mathbf{r}}}$ converging to some tensor $T$. We need to show that also $T$ is in $MS_{\mathcal{F}_{\mathbf{r}}}$. We proceed by showing that there exists a subsequence $T^{(n_j)}$ of the $T^{(n)}$ for which scale-wise convergence $T^{(n_j)}_k \rightarrow T_k$ holds, for some tensors $T_k$ in $\mathcal{F}_{r_k}$ and all $k=0, \ldots , L$. We then show that $T = \sum_{k=0}^L \text{ext}_{L-k}(T_k)$, which means that $T$ is indeed in $MS_{\mathcal{F}_{\mathbf{r}}}$.

For $n$ large enough, it follows that $\|T^{(n)}\| \leq \|T\| + 1$, so $\|T^{(n)}_k\| \leq C\|T\| + C$. For each fixed $k$, the sequence $\{T^{(n)}_k\}_{n=1}^\infty$ is then bounded, so has a convergent subsequence, by the Bolzano-Weierstrass theorem. By passing to subsequences of this subsequence, for each $k$ in turn, it follows that there is a subsequence such that $T^{(n_j)}_k \rightarrow T_k$ for each $k$, for some $T_k$ in $\mathcal{F}_{r_k}$, by closedness of $\mathcal{F}$. It then holds that $T^{(n_j)} = \sum_{k=0}^L \text{ext}_{L-k}(T_k^{(n_j)}) \rightarrow \sum_{k=0}^L \text{ext}_{L-k}(T_k)$. Since by assumption $T^{(n_j)}\rightarrow T$, we must have $T = \sum_{k=0}^L \text{ext}_{L-k}(T_k)$, so $T$ is in $MS_{\mathcal{F}_{\mathbf{r}}}$.
\end{proof}

For this reason, and for reasons to do with the decomposition algorithm presented in the next section, we will mostly restrict to closed tensor formats $\mathcal{F}$ in practice. The tensor-train format is one good candidate for this purpose.

\section{Alternating decomposition algorithm}\label{sec:alg}
This section describes a simple algorithm for computing an approximation of a tensor $T$ in $MS_{\mathcal{F}_{\mathbf{r}}}$. Because of Thm.~\ref{thm:closed}, this approximation problem is ill-posed even when $\mathcal{F}$ is a closed tensor-format. It will therefore not be possible to compute an optimal approximation of $T$ in $MS_{\mathcal{F}_{\mathbf{r}}}$, since it might not even exist. We therefore describe an alternating algorithm, which improves the approximation in every iteration. The tensor $T$ can be given either in full format, or as an already compressed tensor in $\mathcal{F}$. The steps in the algorithm carry through for any weakly closed tensor format $\mathcal{F}$, and the tensor-train format is a good example. The following Lemma will be important for the approximation algorithm.
\begin{lemma}\label{lemma:updownTT}
~\begin{enumerate}
\item If $T$ in $\mathbb{R}^{b_s^m \times \ldots \times b_s^m}$ has canonical decomposition $T = \sum_{k=1}^r \bigotimes_{j=1}^d u_{kj}$ with each $u_{kj} \in \mathbb{R}^{b_s^m}$, then
\begin{equation}
\begin{dcases}
\text{ext}_{\ell}(T) = \sum_{k=1}^r \bigotimes_{j=1}^d \text{ext}_\ell(u_{kj}), \\
\text{ave}_{\ell}(T) = \sum_{k=1}^r \bigotimes_{j=1}^d \text{ave}_\ell(u_{kj})
\end{dcases}
\end{equation} 

\item If $T$ has a tensor-train representation
\begin{equation}
T(i_1, \ldots , i_d) =  \sum_{\alpha_1 =1}^{r_1} \cdots  \!\!\! \sum_{\alpha_{d-1} = 1}^{r_{d-1}}G_1(i_1, \alpha_1)\cdot G_2(\alpha_1, i_2, \alpha_2) \cdot \ldots \cdot G_d(\alpha_{d-1}, i_d),
\end{equation}
with $G_k \in \mathbb{R}^{r_{k-1}\times b_s^k \times r_k}$, then $\text{ext}_\ell(T)$ has a tensor train decomposition with cores $\text{ext}_\ell(G_k(\cdot, i_k, \cdot)) \in \mathbb{R}^{r_{k-1}\times b_s^{k+\ell} \times r_k} $. Similarly, $\text{ave}_\ell(T)$ has a tensor train decomposition with cores $\text{ave}_\ell(G_k(\cdot, i_k, \cdot)) \in \mathbb{R}^{r_{k-1}\times b_s^{k-\ell} \times r_k}$.
\end{enumerate}
\end{lemma}
\begin{proof}
Clear from definitions of $\text{ext}_\ell$ and $\text{ave}_\ell$.
\end{proof}
Note that the cost of computing $\text{ext}_1$ and $\text{ave}_1$ in the tensor-train format is $\mathcal{O}(dnr^2)$, if $r_k \leq r$ and $n_k \leq n$, for all $k = 1, \ldots , d$.

The alternating algorithm starts with an initial approximation $\sum_{k=0}^L \text{ext}_{L-k}(T_k^{(0)})$ to $T$. It proceeds by fixing all scales except for the $k$:th one, and improving the approximation on the $k$:th scale by the following update equation
\begin{equation}
T_k^{(n)} = \argmin_{S \in \mathcal{F}_{r_k}} \| T - \sum_{\ell < k} \text{ext}_{L-\ell}(T_\ell^{(n)}) - \sum_{\ell > k} \text{ext}_{L-\ell}(T_\ell^{(n-1)}) - \text{ext}_{L-k}(S)  \| .
\end{equation}
Sweeping over all indices $k = 0, \ldots , L$ in turn completes one step of the iteration, which is repeated subsequently. In order to compute the updates on each scale, we will use the following result.
\begin{lemma}\label{lemma:updateS}
\begin{equation}
\begin{split}
&\argmin_{S \in \mathcal{F}_{r_k}} \| T - \sum_{\ell < k} \text{ext}_{L-\ell}(T_\ell^{(n)}) - \sum_{\ell > k} \text{ext}_{L-\ell}(T_\ell^{(n-1)}) - \text{ext}_{L-k}(S)  \| \\
&= \argmin_{S \in \mathcal{F}_{r_k}} \| \text{ave}_{L-k}\left(T - \sum_{\ell < k} \text{ext}_{L-\ell}(T_\ell^{(n)}) - \sum_{\ell > k} \text{ext}_{L-\ell}(T_\ell^{(n-1)}) \right)- S  \| \\
&= \argmin_{S \in \mathcal{F}_{r_k}} \| \text{ave}_{L-k}( T )- \sum_{\ell < k} \text{ext}_{k-\ell }(T_\ell^{(n)})  - \sum_{\ell > k} \text{ave}_{\ell - k}(T_\ell^{(n-1)})- S  \|.
\end{split}
\end{equation}
\end{lemma}
\begin{proof}
We show the first equality in the statement; the second follows from the first together with Eq.~\eqref{eq:extprop}. For any $S \in \mathbb{R}^{b_s^k \times \ldots \times b_s^k}$ and any fixed $A \in \mathbb{R}^{b_s^L \times \ldots \times b_s^L}$, we have
\begin{equation}\label{eq:minS}
\begin{split}
\| A - \text{ext}_{L-k}(S)  \|^2 &= \|A\|^2 - 2\langle A, \text{ext}_{L-k}(S) \rangle + \|\text{ext}_{L-k}(S)\|^2 \\
&= \|A\|^2 - 2b_s^{d(L-k)}\langle \text{ave}_{L-k}(A), S \rangle + b_s^{d(L-k)}\|S\|^2,
\end{split}
\end{equation}
so a minimizer of Eq.~\eqref{eq:minS} is also a minimizer $- 2\langle \text{ave}_{L-k}(A), S \rangle + \|S\|^2$. Since $A$ is fixed, $S$ is therefore also a minimizer of
\begin{equation}
\|\text{ave}_{L-k}(A)\|^2 - 2\langle \text{ave}_{L-k}(A), S \rangle + \|S\|^2 = \| \text{ave}_{L-k}(A) - S  \|^2,
\end{equation}
from which the statement follows when taking
\begin{equation}
A = T - \sum_{\ell < k} \text{ext}_{L-\ell}(T_\ell^{(n)}) - \sum_{\ell > k} \text{ext}_{L-\ell}(T_\ell^{(n-1)}).
\end{equation}
\end{proof}

Lemma~\ref{lemma:updateS} implies that the approximation $T^{(n)}_k$ on each scale can be updated by solving an optimal approximation problem in $\mathcal{F}_{r_k}$. Since $\mathcal{F}$ was assumed weakly closed, this problem is well-posed. In practice, it might be computationally easier to instead find a quasi-optimal approximation for this update. This is for instance the case when using the tensor-train approximation, where the standard TT-SVD algorithm \cite[Alg.~1]{oseledets2011tensor} indeed only guarantees a quasi-optimal result.

We can structure the update steps in a downward and upward sweep to prevent recalculating the tensors on the coarser grids. In the downward sweep, we calculate and store
\begin{equation}
\begin{split}
T_{\text{down}, L} &= T, \\
T_{\text{down}, k-1} &= \text{ave}_1(T_{\text{down},k}  - T_k^{(n-1)}), \text{ for }k = L, \ldots, 1. 
\end{split}
\end{equation}
In closed form, this results in
\begin{equation}
T_{\text{down}, k} = \text{ave}_{L-k}\left(T - \sum_{\ell > k} \text{ext}_{L-\ell}(T_\ell^{(n-1)})\right),
\end{equation}
so the update equation on the $0$:th level in Lemma~\ref{lemma:updateS} reads as 
\begin{equation}
T_0^{(n)} = \text{round}_{\mathcal{F}_{r_{0}}}(T_{\text{down},0}).
\end{equation}
Next, in the upward sweep, we set
\begin{equation}
\begin{split}
T_{\text{up}, 0} &= 0, \\
T_{\text{up},k} &= \text{ext}_1(T_{\text{up},k-1} + T_k^{(n)} ), \text{ for } k = 1, \ldots, L-1,
\end{split}
\end{equation}
where $T_k^{(n)}$ is calculated from $T_k^{(n-1)}$ and $T_{\text{up},k-1}$ by the following procedure. We have
\begin{equation}
T_{\text{down}, k} - T_{\text{up}, k-1} = \text{ave}_{L-k}( T )- \sum_{\ell < k} \text{ext}_{k-\ell }(T_\ell^{(n)})  - \sum_{\ell > k} \text{ave}_{\ell - k}(T_\ell^{(n-1)}),
\end{equation}
so the update equation in Lemma~\ref{lemma:updateS} reads as 
\begin{equation}
T_k^{(n)} = \text{round}_{\mathcal{F}_{r_k}}(T_{\text{down}, k} - T_{\text{up}, k-1}).
\end{equation}
This procedure is summarized in Alg.~\ref{alg:AL-multi}.

\begin{algorithm}
\caption{Alternating multiresolution decomposition}\label{alg:AL-multi}
\begin{algorithmic}[1]
 \Require{$d$-tensor $T$ in full format or $\mathcal{F}$, vector of rank bounds $(r_0, \ldots , r_L)$, maximum number of iterations $M$.}{}
 \Ensure{Approximation $\sum_{k=0}^L \text{ext}_k(T_k)$ to $T$.}{}
\For{$k =0:L$}\Comment{Initialization}
\State $T_k^{(0)} = 0$
\EndFor
\For{$n =1:M$}\Comment{Main loop}
\State $T_{\text{down}, L} = T$
\For{$k =L:-1:1$}\Comment{Downward sweep}
\State         $T_{\text{down}, k-1} = \text{ave}_1(T_{\text{down},k}  - T_k^{(n-1)})$
\EndFor     
\State $T_{\text{up}} = 0$
\For{$k =0:L-1$}\Comment{Upward sweep}
\State	$T_k^{(n)} = \text{round}_{\mathcal{F}_{r_k}}(T_{\text{down}, k} - T_{\text{up}})$
\State        $T_{\text{up}} = \text{ext}_1(T_{\text{up}} + T_k^{(n)} )$
\EndFor     
\State	$T_L^{(n)} = \text{round}_{\mathcal{F}_{r_L}}(T - T_{\text{up}})$
\EndFor
 \end{algorithmic}
\end{algorithm}

In the case when $T$ is already given in the tensor-train format, Lemma~\ref{lemma:updownTT} shows that $\text{ext}_1(T)$ and $\text{ave}_1(T)$ can be computed in the compressed format with cost $\mathcal{O}(r^2dn)$, where $r$ is the maximum of the TT-ranks of $T$. Since the tensors $T_{\text{down}, k} = \text{ave}_{L-k}(T - \sum_{\ell > k} \text{ext}_{L-\ell}(T_\ell^{(n)}))$ appearing in the downward sweep have TT-representation with rank at most $R = r + r_0 + \ldots + r_L$, the downward sweep therefore has cost bounded by
\begin{equation}
\mathcal{O}(R^2d(1 + b_s + \ldots + b_s^L)) = \mathcal{O}(R^2d\frac{ b_s^L - 1}{b_s - 1}) = \mathcal{O}(R^2db_s^L) = \mathcal{O}(R^2dn),
\end{equation}

In the upward sweep, a quasi-optimal minimizer of $\argmin_{S \in \mathcal{F}_{r_k}} \| T_{\text{down},k} - T_{\text{up},k}  - S\|$ can be computed by a call to the TT-rounding procedure. Since $T_{\text{down},k} - T_{\text{up},k}$ has TT-ranks at most $r+\sum_{\ell \neq k}r_k$, one iteration of Alg.~\ref{alg:AL-multi} is then of cost at most
\begin{equation}
\mathcal{O}(R^3d(1 + b_s + \ldots + b_s^L)) = \mathcal{O}(R^3d\frac{ b_s^L - 1}{b_s - 1}) = \mathcal{O}(R^3db_s^L) = \mathcal{O}(R^3dn).
\end{equation}
The total cost of Alg.~\ref{alg:AL-multi} is therefore $\mathcal{O}(MR^3dn)$.

In general, alternating algorithms of the form in Alg.~\ref{alg:AL-multi} lead to a monotonically decreasing objective function $\|T - \sum_{k=0}^L \text{ext}_{L-k}(T_k^{(n)})\|$. However, there are in general no guarantees that the approximations $T_k^{(n)}$ on each scale converge, and even if they do, convergence might occur to only a local minimum. This is a typical situation when dealing with tensors in dimension higher than two, and occurs for instance when using the popular alternating least-squares algorithm for computing a (non-multiresolution) canonical decomposition of a tensor \cite{uschmajew2012local,wang2014global}, and when using iterative methods for computing canonical decompositions with orthogonality constraints \cite{chen2009tensor,wang2015orthogonal}. Appendix~\ref{sec:local} states and proves a local convergence guarantee for Alg.~\ref{alg:AL-multi}.

\section{Tensor operations}\label{sec:operations}
The different scales of the multiresolution format introduce a grading on $MS_{\mathcal{F}_\mathbf{r}}$, and we now show that all common tensor operations can be performed in such a way that they respect the graded structure and can be computed with cost independent of the number of levels $L$. The format can therefore be used in calculations without having to convert into full format.

\subsection{Addition} If $T$ and $S$ have the multiresolution representations $(T_0, \ldots , T_L)$ and $(S_0, \ldots , S_L)$, respectively, then $S+T$ has multiresolution representation $(T_0+S_0, \ldots , T_L+S_L)$.

\subsection{Rounding} Let $T$ have multiresolution representation $(T_0, \ldots , T_L)$ with each $T_k \in \mathcal{F}_{r_k}$. When the $T_k$ potentially have suboptimal ranks, for instance as a result of having performed addition or taking Hadamard products, a multiresolution representation with more beneficial ranks is given by $(\widetilde{T}_0, \ldots , \widetilde{T}_L)$, with $\widetilde{T}_k = \text{round}_\mathcal{F}(T_k, b_s^{-d(L-k)}\varepsilon)$. This results in an approximation error
\begin{equation}
\|T - \sum_{k=0}^{L}\text{ext}_{L-k}(\widetilde{T}_k)\| \leq \sum_{k=0}^{L} \| \text{ext}_{L-k} (T_k - \widetilde{T}_k) \| \leq \varepsilon \sum_{k=0}^L \|T_k\|.
\end{equation}

The cost of this procedure is given by
\begin{equation}
\sum_{k=0}^L \mathcal{O}(r^3b_s^kd) =  \mathcal{O}(r^3d \frac{b_s^{L+1}-1}{b_s - 1})  = \mathcal{O}(r^3d b_s^{L}) =  \mathcal{O}(r^3d n).
\end{equation}

\subsection{Hadamard product} If $T$ and $S$ have the respective multiresolution representations $(T_0, \ldots , T_L)$ and $(S_0, \ldots , S_L)$, then $S\circ T$ has multiresolution representation $(R_0, \ldots , R_L)$ with
\begin{equation}
R_k = T_k\circ (\text{ext}_k(S_0) + \ldots + S_{k}) + S_k \circ (\text{ext}_k(T_0) + \ldots + \text{ext}_1(T_{k-1})).
\end{equation}
 $R_k$ can be computed recursively, with rounding during intermediate steps to avoid rank-growth, i.e., $A_0 = T_0, A_k = \text{round}_{\mathcal{F}}(T_k + \text{ext}_1(A_{k-1}))$ and $B_0 = S_0, B_k = \text{round}_{\mathcal{F}}(S_k + \text{ext}_1(B_{k-1}))$. This results in
\begin{equation}
R_k = T_k\circ A_k + S_k\circ B_k.
\end{equation}
In the case when $\mathcal{F}$ is the tensor-train format, $\text{ext}_1(T_{k-1})$ can be computed with cost $\mathcal{O}(\text{rank}_{TT}(A_{k-1})^2b_s^kd)$, by Lemma~\ref{lemma:updownTT}, and the rounding procedure has cost $\mathcal{O}(\text{rank}_{TT}(A_{k-1}+T_{k})^3b_s^kd)$. If we write
\begin{equation}
R = \max_{1\leq k \leq L} \text{rank}_{TT}(\text{ext}_k(T_0) + \text{ext}_{k-1}(T_1) + \ldots + T_k) \leq r_0 + \ldots + r_L,
\end{equation}
it follows that the total cost of computing the Hadamard product is
\begin{equation}
\sum_{k=0}^L \mathcal{O}(R^3b_s^kd) =  \mathcal{O}(R^3d \frac{b_s^{L+1}-1}{b_s - 1})  = \mathcal{O}(R^3d b_s^{L}) =  \mathcal{O}(R^3d n).
\end{equation}

\subsection{Tensor-vector contraction} If the tensor $T$ has a multiresolution representation $(T_0, \ldots , T_L)$ then $T\times_j v$ has multiresolution representation
\begin{equation}
(b_s^L T_0\times_j \text{ave}_L(v), b_s^{L-1}T_1\times_j \text{ave}_{L-1}(v), \ldots , T_L \times_j v).
\end{equation}
The cost of computing $T_k \times_j v$ when $T_k$ is given in the tensor-train format, is $\mathcal{O}(r_k^2dn)$. Recursively computing $\text{ave}_{k}(v) = \text{ave}(\text{ave}_{k-1}(v))$ has cost $\mathcal{O}(n)$. The total cost then becomes
\begin{equation}
\sum_{k=0}^L \mathcal{O}(r^2b_s^kd) =  \mathcal{O}(r^2d \frac{b_s^{L+1}-1}{b_s - 1})  = \mathcal{O}(r^2d b_s^{L}) =  \mathcal{O}(r^2d n).
\end{equation}

\subsection{Frobenius norm} The Frobenius norm can be computed as $(T\circ T)\times_1 v \ldots \times_d v$, where $v \in \mathbb{R}^n$ is a vector with all entries equal to $1$.

\section{Applications}\label{sec:applications}
This section compares the compression ratios achieved using Alg.~\ref{alg:AL-multi} to those of using the tensor-train decomposition, for a variety or real-world datasets. In $2$D, we show how this can be used to achieve greater accuracy than a truncated singular value decomposition for given storage. In dimensions higher than two, we achieve greater compression than the tensor-train decomposition.

All computations were carried out on a MacBook Pro with a 3.1 GHz Intel Core i5 processor and 16 GB of memory.

\subsection{Motivating example revisited}
We consider the tensor $T \in \mathbb{R}^{n\times n \times n}$ in Eq.~\eqref{eq:motivating_example} in the motivating example. In the canonical format, we compare the approximation error of a rank-$2$ approximation, obtained by the standard alternating least-squares algorithm \cite{TTB_Software,TTB_Dense,TTB_Sparse}, to the multiresolution canonical approximation produced by Alg.~\ref{alg:AL-multi}. We used rank vector $(0, \ldots , 0, 1, 1, 1)$ which then has lower storage cost than the rank-$2$ approximation, with a single iteration of Alg.~\ref{alg:AL-multi}. For the alternating least-squares algorithm on each scale, we used the HOSVD as initial guess. The result is shown in Fig.~\ref{fig:ex_motivating}. Consistent with Sec.~\ref{sec:motivating_example}, Alg.~\ref{alg:AL-multi} produces an approximation with relative error scaling as $\mathcal{O}(n^{-1})$. Alg.~\ref{alg:AL-multi} therefore results in far lower approximation error for the same compression ratio, provided $n$ is large enough.
\begin{figure}
\begin{center}
 \includegraphics[width=0.96\textwidth]{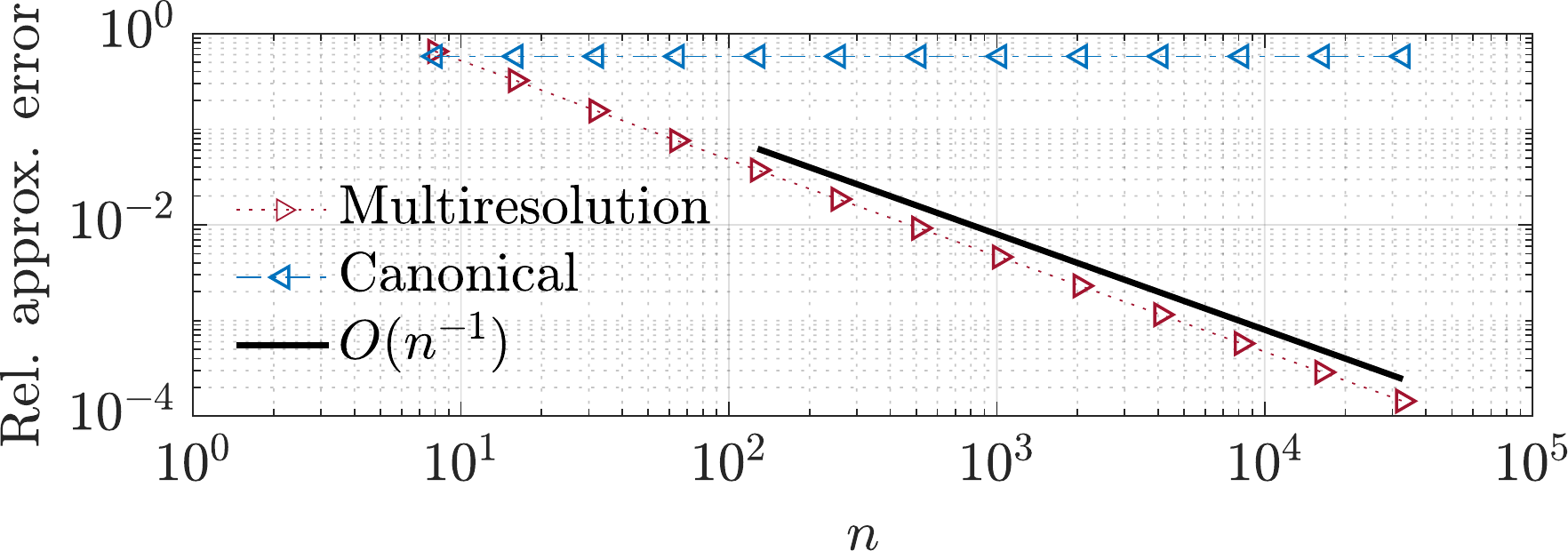} 
\caption{Compression of the tensor in Eq.~\eqref{eq:motivating_example} as a rank-$2$ approximation in the canonical format and in the multiresolution canonical format. The relative approximation error is shown as a function of the tensor dimension $n$.}
\label{fig:ex_motivating}
\end{center}
\end{figure}

\subsection{Image data}

We consider two matrices with multiscale features. The matrices are greyscale versions of images with features on several scales \cite{nycphoto,earthphoto}, rescaled to be of size $2048\times 2048$ pixels. The resulting matrices in $\mathbb{R}^{2048\times 2048}$ were compressed in both the low-rank matrix format and in the multiresolution format. We used batch-size $b_s = 2$ and rank vector $(r, \ldots , r)$ for increasing values of $r$ and maximum number of iterations $M$. The results are shown in Fig.~\ref{fig:ex_matrix_results}. For accuracies for which the low-rank matrix format achieves a compression ratio of at least two, the multiresolution format achieves up to a factor $1.5$ higher compression ratio.

\begin{figure}
\begin{tabular}{c}
  \includegraphics[width=0.96\textwidth]{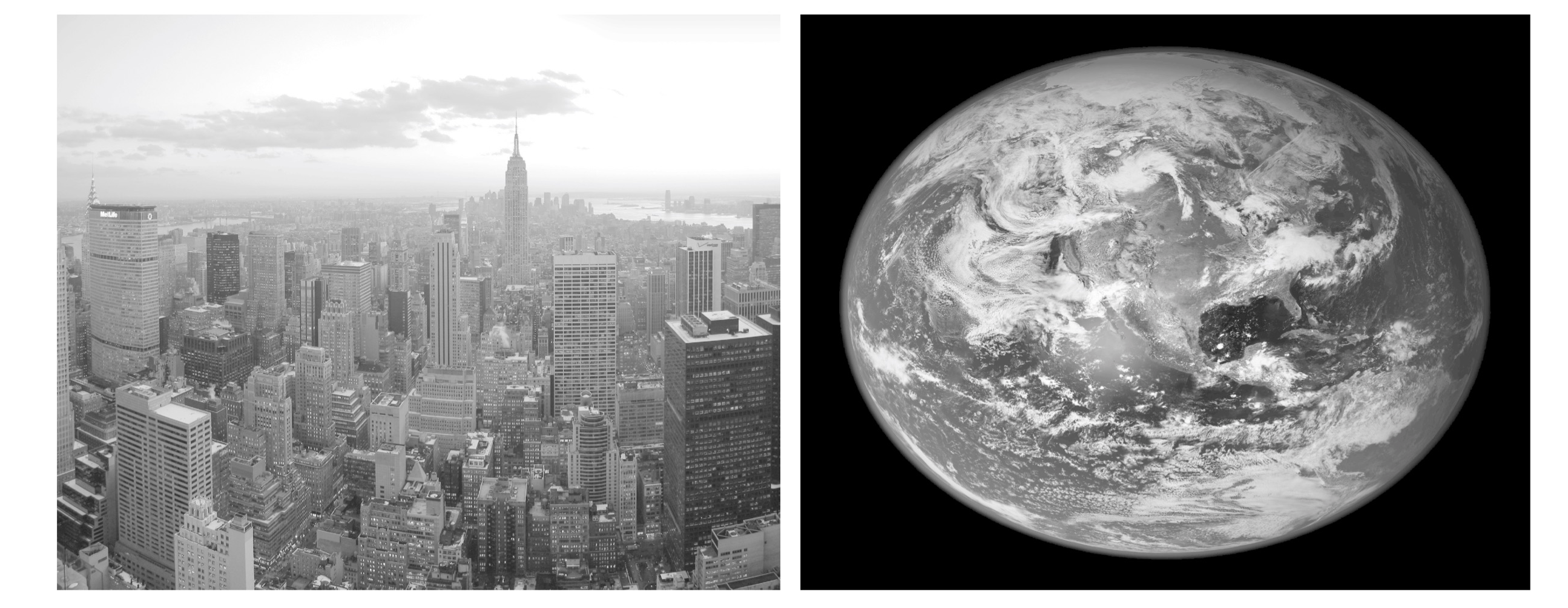}  \\
  \includegraphics[width=0.96\textwidth]{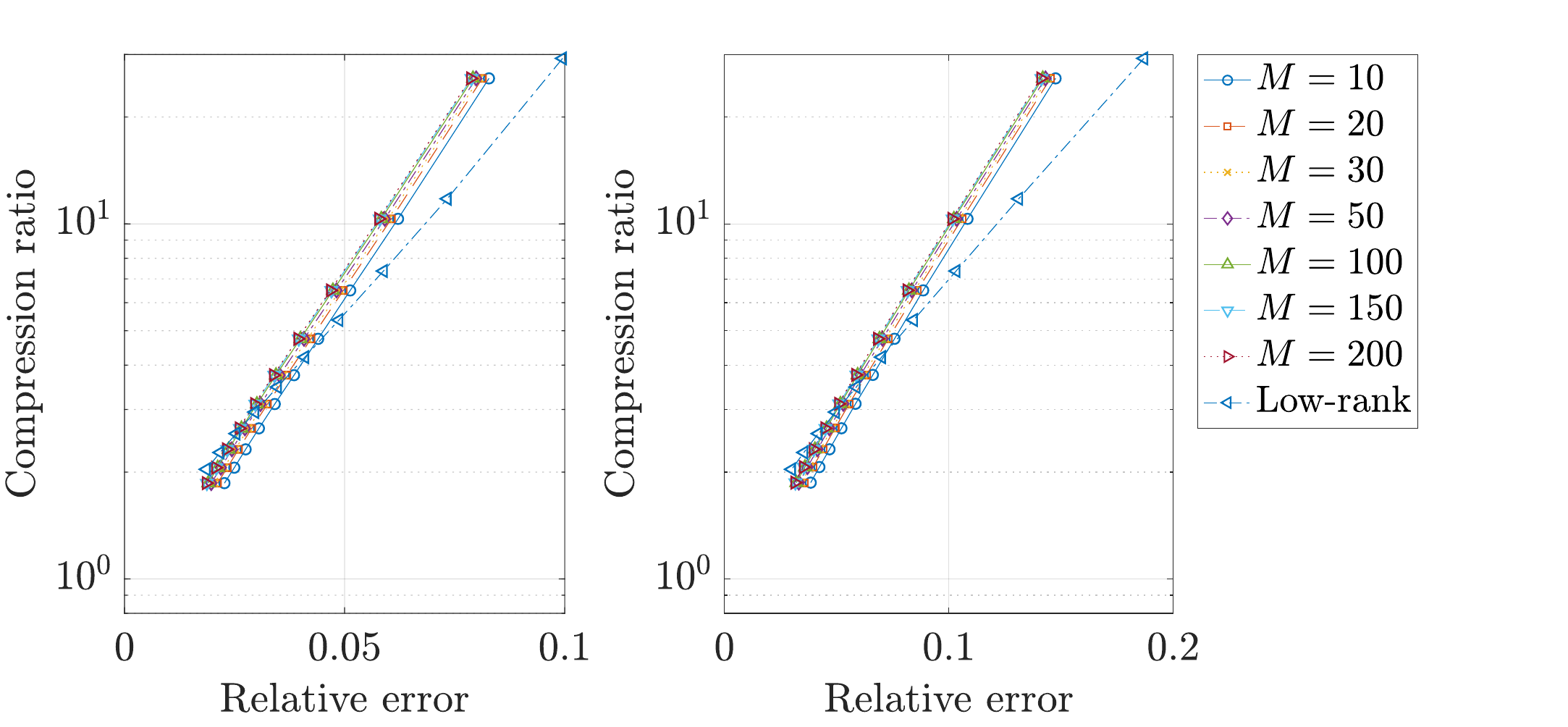}  \\
  \includegraphics[width=0.96\textwidth]{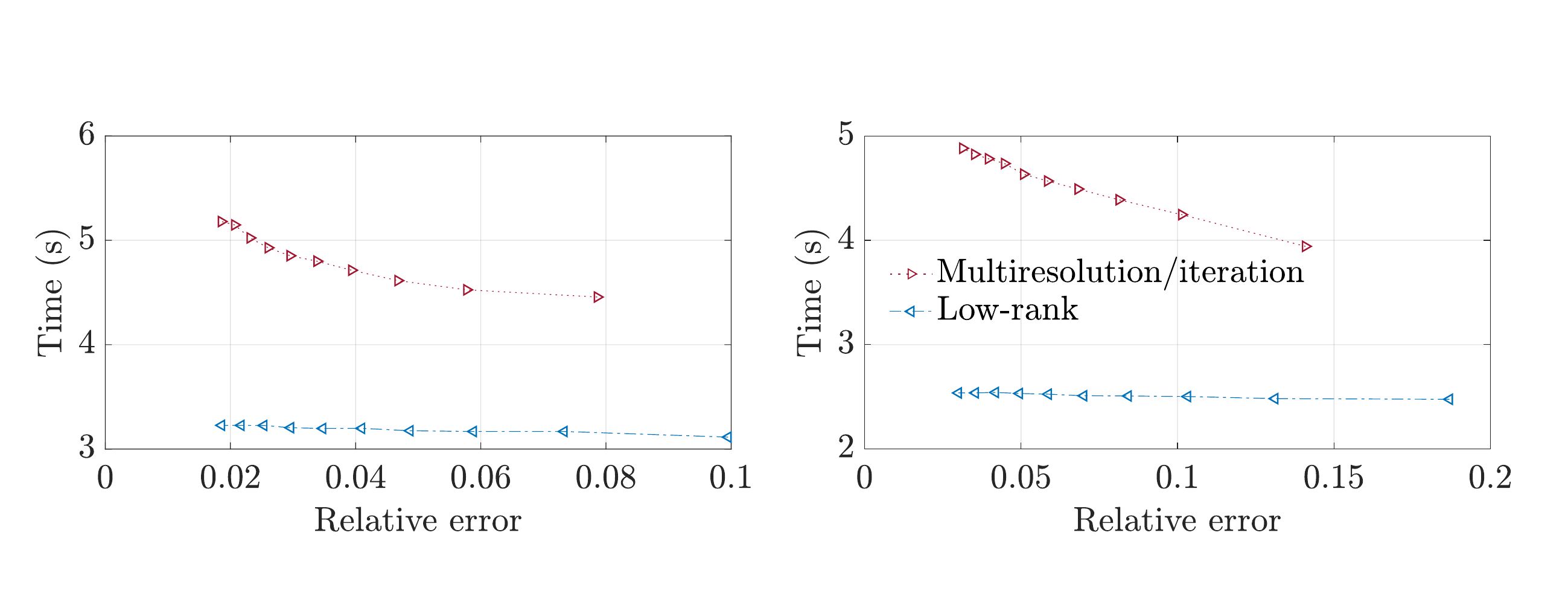}  \\
\end{tabular}
\caption{Top row: Images from \cite{nycphoto}, \cite{earthphoto}, respectively, Middle row: compression ratios of matrices in \cite{nycphoto,earthphoto} as functions of approximation error.  Bottom row: runtimes as functions of approximation error.} \label{fig:ex_matrix_results}
\end{figure}

Fig.~\ref{fig:ex_matrix_show} shows a side-by-side comparison of one of the images compressed in both the multiresolution low-rank matrix format, and the ordinary low-rank matrix format. For the same compression ratio, the multiresolution format has visibly significantly clearer features and correspondingly lower approximation error.

\begin{figure}
\begin{tabular}{c}
  \includegraphics[width=0.96\textwidth]{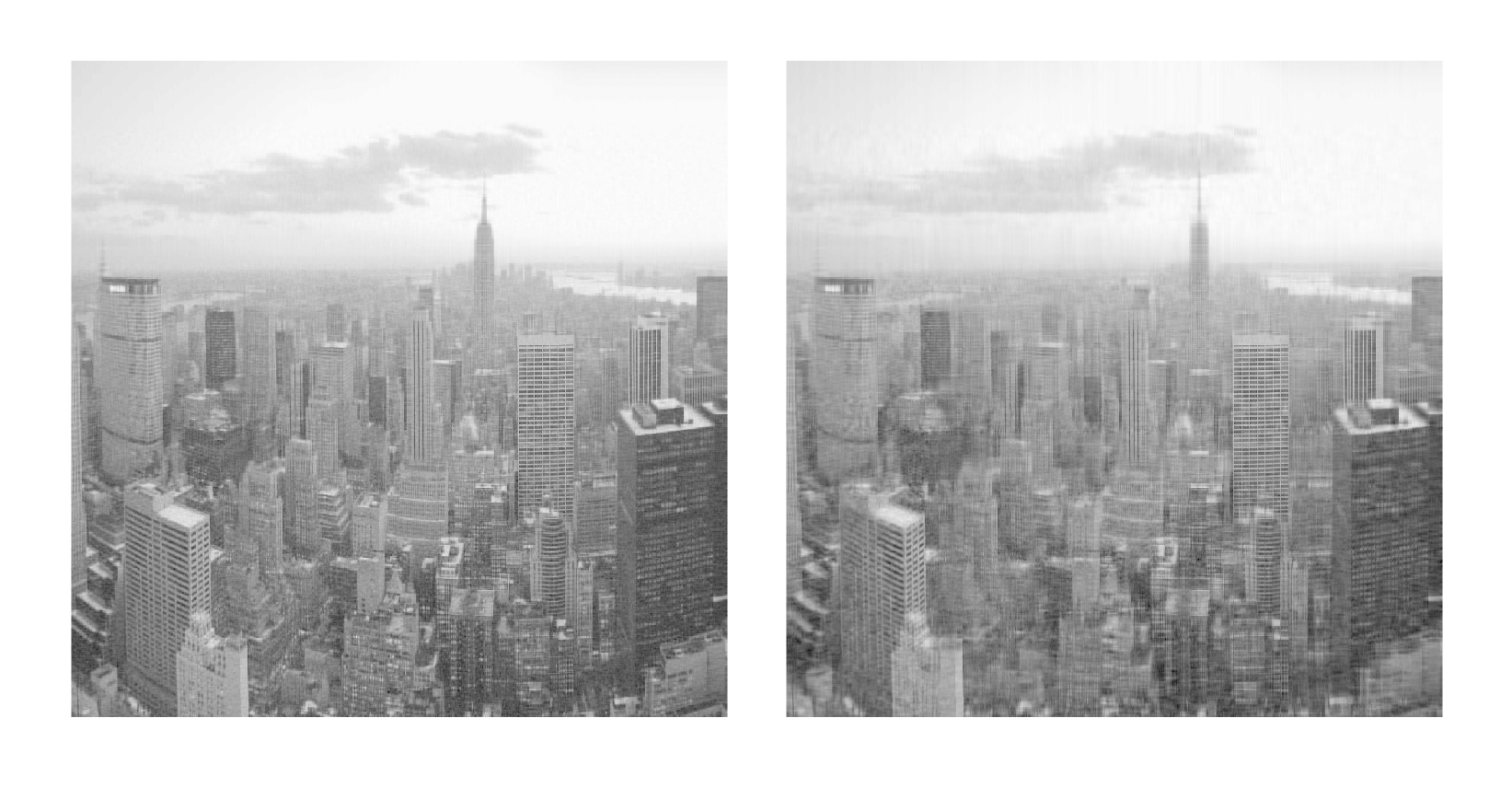}
\end{tabular}
\caption{Compressed images with equal compression ratio in multiresolution format (left) and low-rank matrix format (right).} \label{fig:ex_matrix_show}
\end{figure}

Figs.~\ref{fig:ex_matrix_first} and \ref{fig:ex_matrix_second} show the compressed version of the matrix $A$, decomposed into its different scales. The $i$:th subimage in Fig.~\ref{fig:ex_matrix_first} contains the sum of the $i$ highest scales of the compressed format, i.e., $\sum_{k=0}^{i-1} \text{ext}_{L-k} \left(A_k\right)$. The $i$:th subimage in Fig.~\ref{fig:ex_matrix_second} contains the sum of the $i$ lowest scales of the compressed format, i.e., the matrix $\sum_{k=L-i+1}^{L} \text{ext}_{L-k} \left(A_k\right)$.
\begin{figure}
\begin{tabular}{c}
\includegraphics[width=0.96\textwidth]{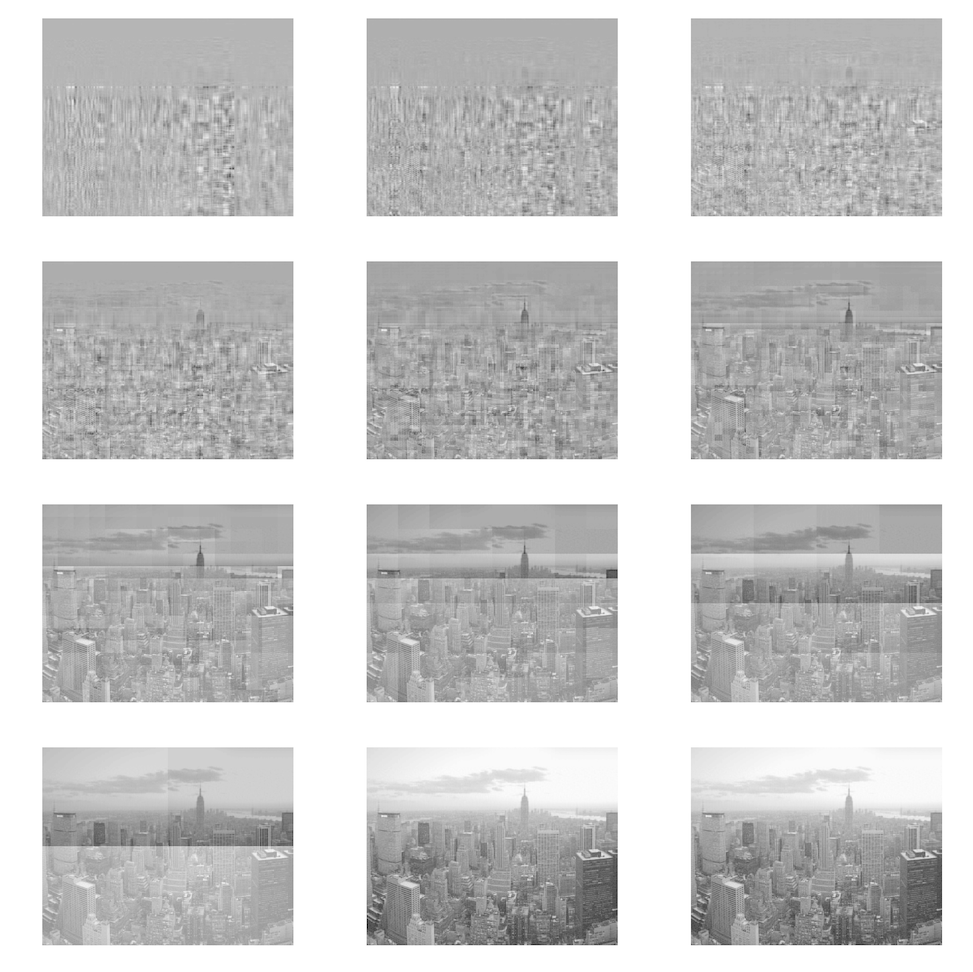}
\end{tabular}
\caption{The $i$th successive image shows the sum of the $i$ highest scales in the compressed multiresolution format.} \label{fig:ex_matrix_first}
\end{figure}

\begin{figure}
\begin{tabular}{c}
 \includegraphics[width=0.96\textwidth]{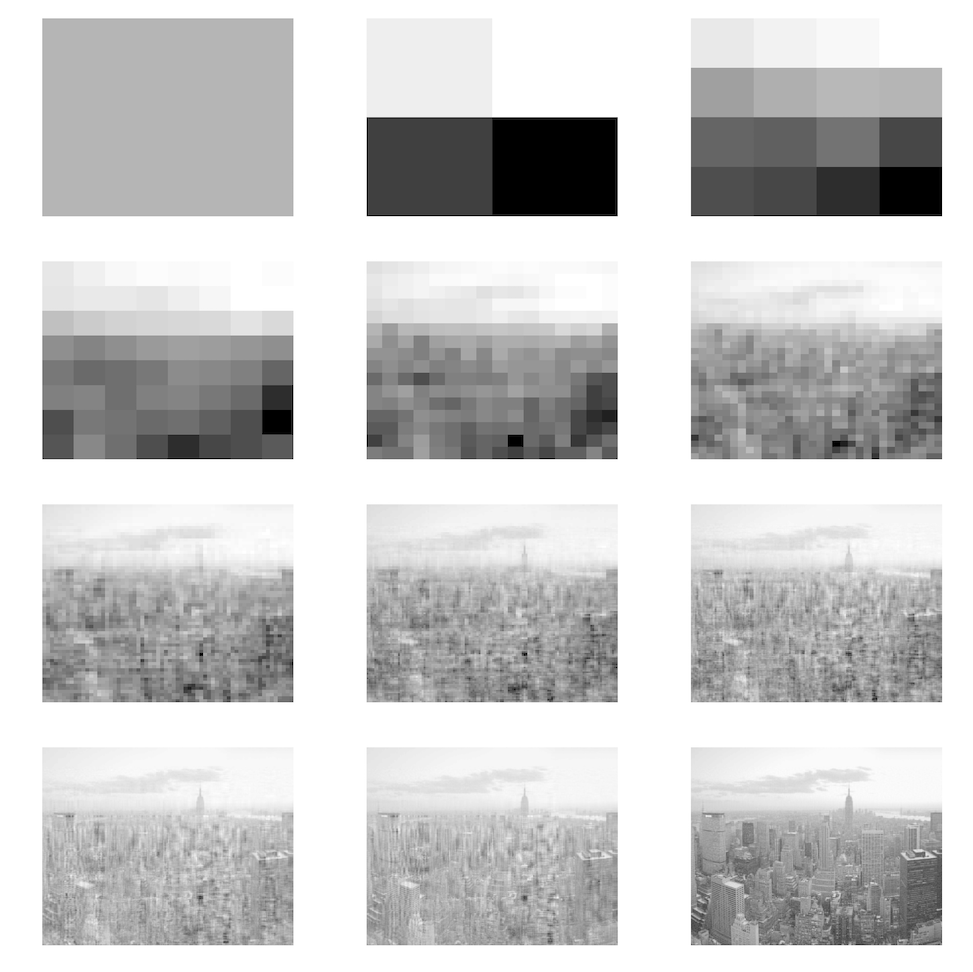}
\end{tabular}
\caption{The $i$th successive image shows the sum of the $i$ lowest scales in the compressed multiresolution format.} \label{fig:ex_matrix_second}
\end{figure}

\subsection{Hyperspectral wavelength data}
We consider hyperspectral wavelength data, which are aerial view photographs of different environments, captured at multiple wavelengths across the electromagnetic spectrum. The data is organized into a tensor $S \in \mathbb{R}^{n_1\times n_2 \times n_3}$ where $n_1$ and $n_2$ are the physical dimensions, and $n_3$ the number of recorded wavelengths. Each slice $T( \cdot, \cdot, i)$ therefore contains a photograph at the $i$:th recorded wavelength.

We consider three different hyperspectral images of different environments \cite{urbandataset,avirisdataset,samsondataset}. One slice of each environment is shown in Fig.~\ref{fig:ex_hyperspectral_results}. For the data from \cite{urbandataset}, slices $1-4$, $76$, $87$, $101-111$, $136-153$ and $198-210$ are removed because of contamination by atmospherical effects, and we consider the first $128$ slices of the upper left $256\times 256$ sub-image. For the data from \cite{avirisdataset}, we consider the $1920\times 640$-pixel subimage starting at the index $(1, 1)$, and for the data from \cite{samsondataset}, we consider the $896\times 896$-pixel subimage starting at index $(1, 1)$, each with $128$ slices. Since each image exhibits features on multiple scales, we expect the multiresolution format to achieve good compression.

This results in three tensors in $\mathbb{R}^{256 \times 256 \times 128},\mathbb{R}^{1920 \times 640 \times 128},\mathbb{R}^{896 \times 896 \times 128}$, respectively, which we compress using the multiresolution tensor-train format and the non-multiresolution tensor-train format. The results are shown in Fig.~\ref{fig:ex_hyperspectral_results}, and shows higher compression ratio in the multiresolution format across practically all accuracies where the tensor-train achieves a compression ratio of at least $1$. The compression ratio is several times larger in the multiresolution tensor format for a wide range of accuracies. The simulations used batch-size $b_s = 2$, rank-vector $(r, \ldots, r)$ for increasing $r$ and maximum number of iterations $M$. The runtime per iteration is a small factor times that of the tensor-train decomposition. To lower the total computational time, one can also consider using randomized algorithms \cite{sun2019low,che2019randomized,wang2015fast} for the tensor approximation on each scale.

\begin{figure}
\begin{tabular}{c}
 \includegraphics[width=0.96\textwidth]{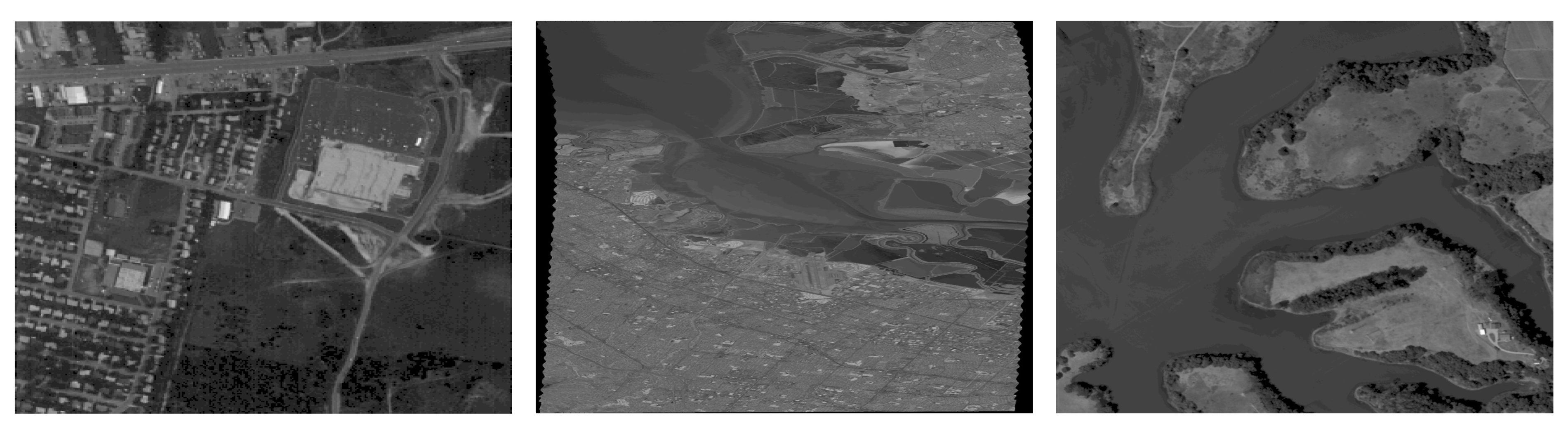} \\
  \includegraphics[width=0.96\textwidth]{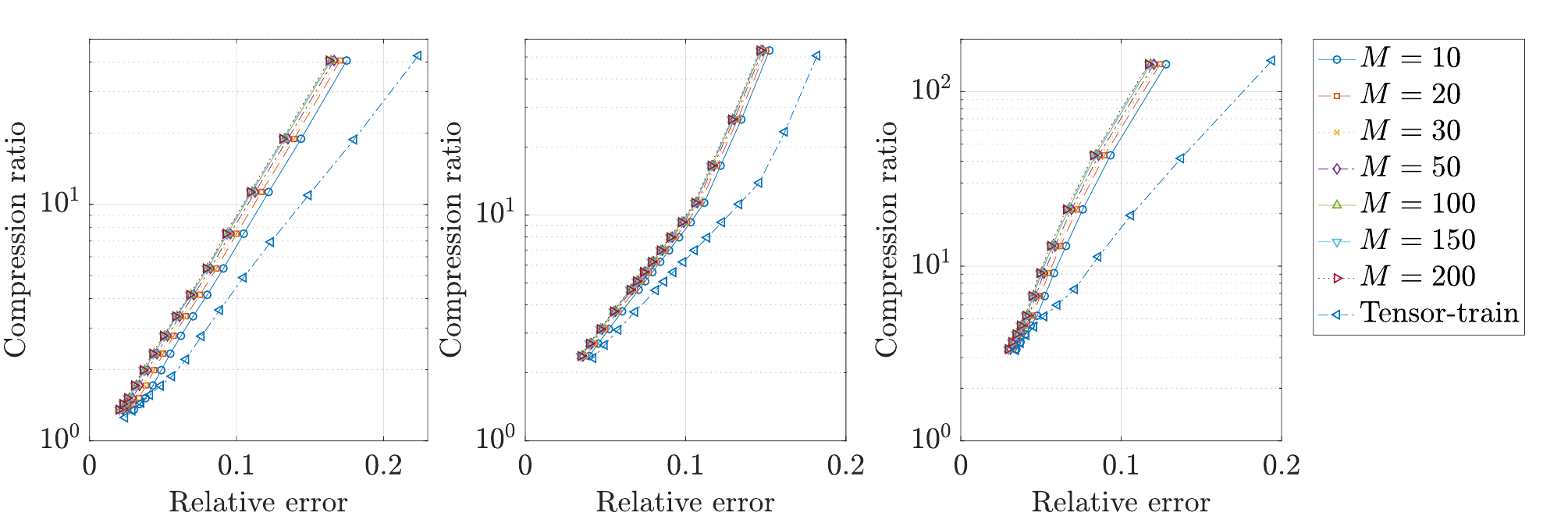}  \\
 \includegraphics[width=0.96\textwidth]{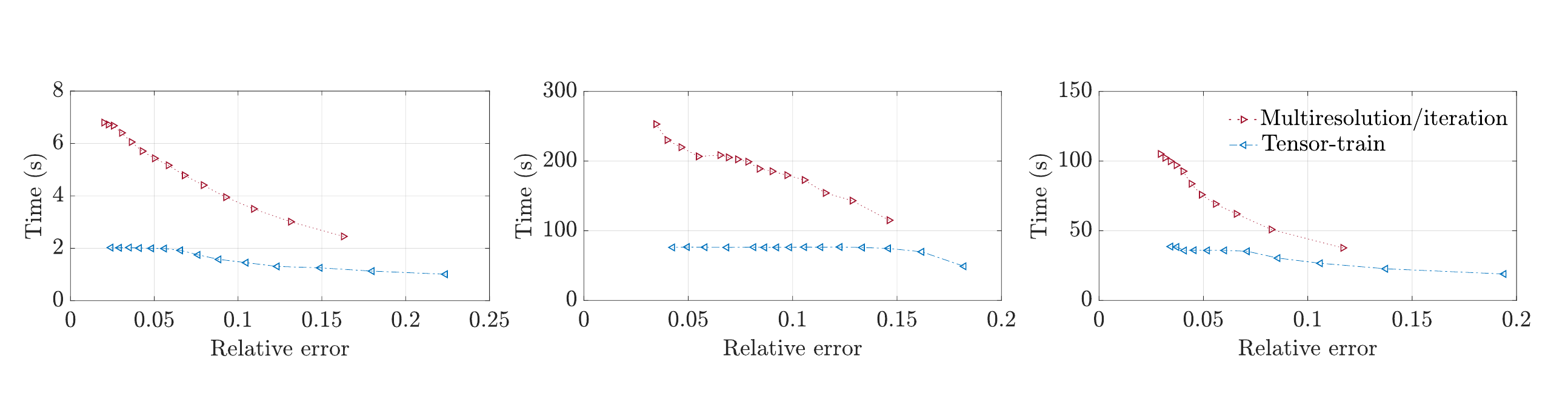}  \\
\end{tabular}
\caption{Top row: One wavelength of the hyperspectral images from \cite{urbandataset}, \cite{avirisdataset}, \cite{samsondataset}, respectively. Middle row: compression ratios of the tensors in \cite{urbandataset}, \cite{avirisdataset}, \cite{samsondataset}, respectively, as functions of approximation error and maximum number of iterations $M$. Bottom row: runtimes as functions of approximation error.} \label{fig:ex_hyperspectral_results}
\end{figure}

\subsection{Video data}
We consider video recordings of three different scenes \cite{li2004statistical}. These correspond to tensors in $\mathbb{R}^{256\times 256 \times 1280}$, $\mathbb{R}^{128\times 128 \times 128}$, and $\mathbb{R}^{128\times 128 \times 1280}$, respectively. The scenes exhibit multiple physical and temporal scales due to e.g., objects moving through the scenes at different speeds. Sample frames are shown together with the compression results in Fig.~\ref{fig:ex_video_results}. We used batch-size $b_s = 2$ and rank vector $(r, \ldots , r)$ for increasing values of $r$ and maximum number of iterations $M$. The multiresolution approximation achieves up to more than twice as high compression ratio as the tensor-train decomposition, over a wide range of accuracies.

\begin{figure}
\begin{tabular}{c}
 \includegraphics[width=0.96\textwidth]{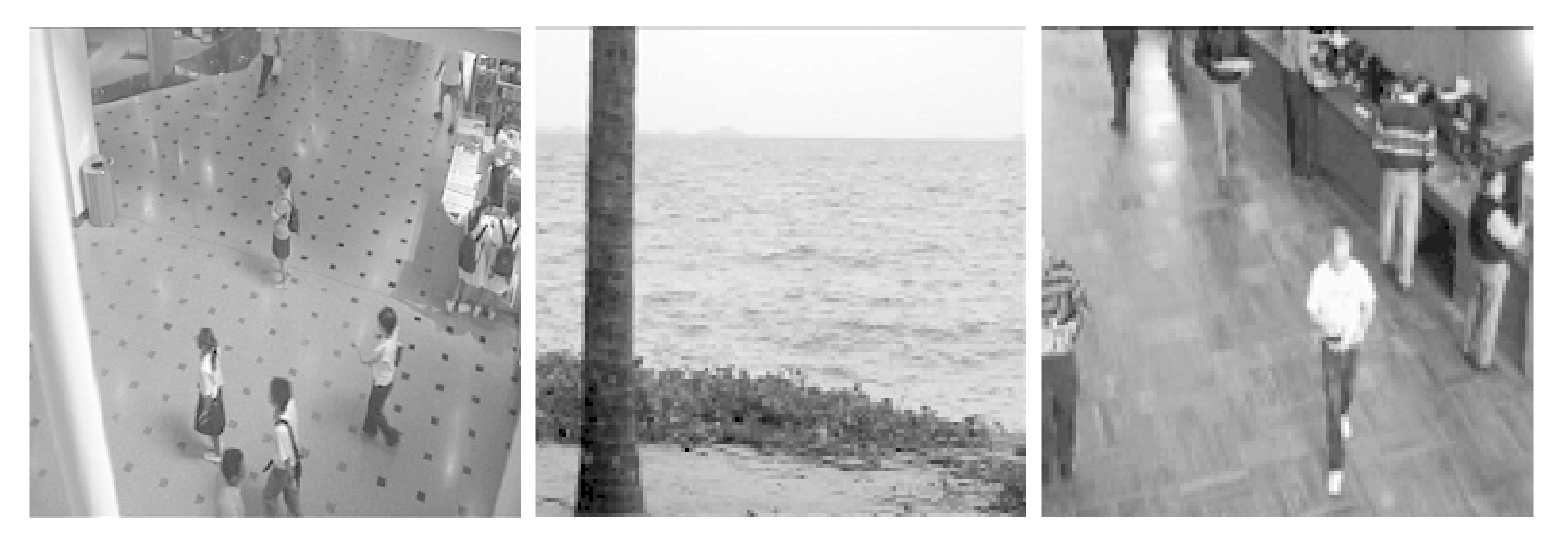} \\
  \includegraphics[width=0.96\textwidth]{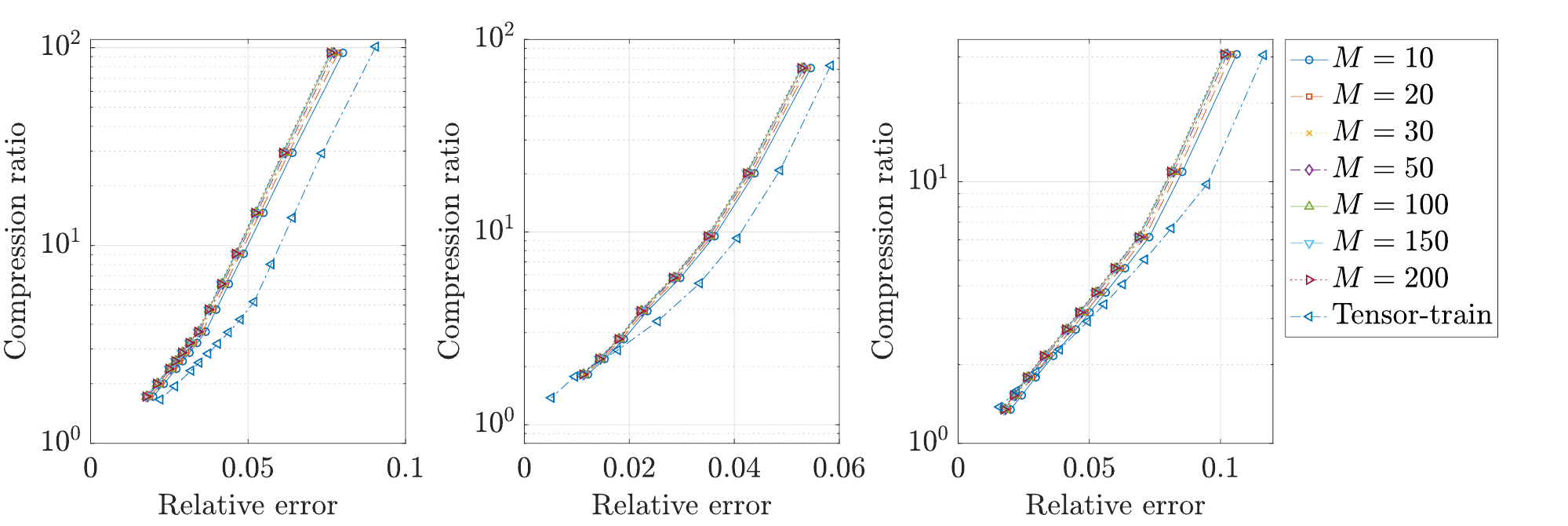}  \\
 \includegraphics[width=0.96\textwidth]{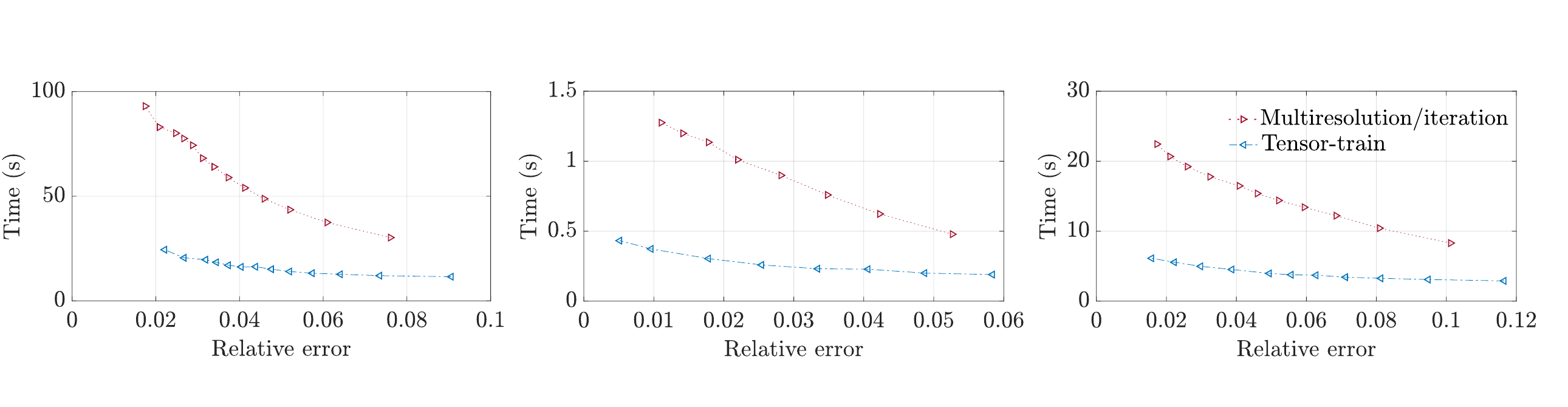}  \\
\end{tabular}
\caption{Top row: One frame of original videos. Middle row: compression ratios of the tensors as functions of approximation error and maximum number of iterations $M$. Bottom row: runtimes as functions of approximation error.} \label{fig:ex_video_results}
\end{figure}

A sample compressed scene is included in Movie~1\footnote{\href{https://github.com/MultiResTF/multiresolution/blob/master/movies/movie1.mp4}{\texttt{https://github.com/MultiResTF/multiresolution/blob/master/movies/movie1.mp4}}}, with a side-by-side comparison of the multiresolution approximation and the tensor-train approximation. The approximations shown achieve the same compression ratio, and the multiresolution approximation exhibits more clearly defined features and noticeably fewer artifacts of the approximation procedure, compared to the tensor-train approximation. Movie~2\footnote{\href{https://github.com/MultiResTF/multiresolution/blob/master/movies/movie2.mp4}{\texttt{https://github.com/MultiResTF/multiresolution/blob/master/movies/movie2.mp4}}} presents the approximation on each scale. The $i$th submovie contains the sum of the $i$ highest scales of the compressed format, i.e., $\sum_{k=0}^{i-1} \text{ext}_{L-k} \left(T_k\right)$.

\subsection{Multiresolution canonical decomposition}\label{sec:ex_video_candecomp}
We conclude this section with an example of compression into the multiresolution format, when the underlying tensor format is the canonical tensor format. The canonical format is not closed and there are no guarantees to find an optimal approximation in the format. However, Alg.~\ref{alg:AL-multi} can still be run with one of the standard approximation algorithms for approximation into each scale. We use the alternating least-squares algorithm \cite{TTB_Software,TTB_Dense,TTB_Sparse} with standard settings, on the first tensor in Fig.~\ref{fig:ex_video_results}. The results are shown in Fig.~\ref{fig:ex_video_candecomp}, with higher compression ratios for a given accuracy, compared to the non-multiresolution format.

\begin{figure}
\begin{center}
 \includegraphics[width=0.9\textwidth]{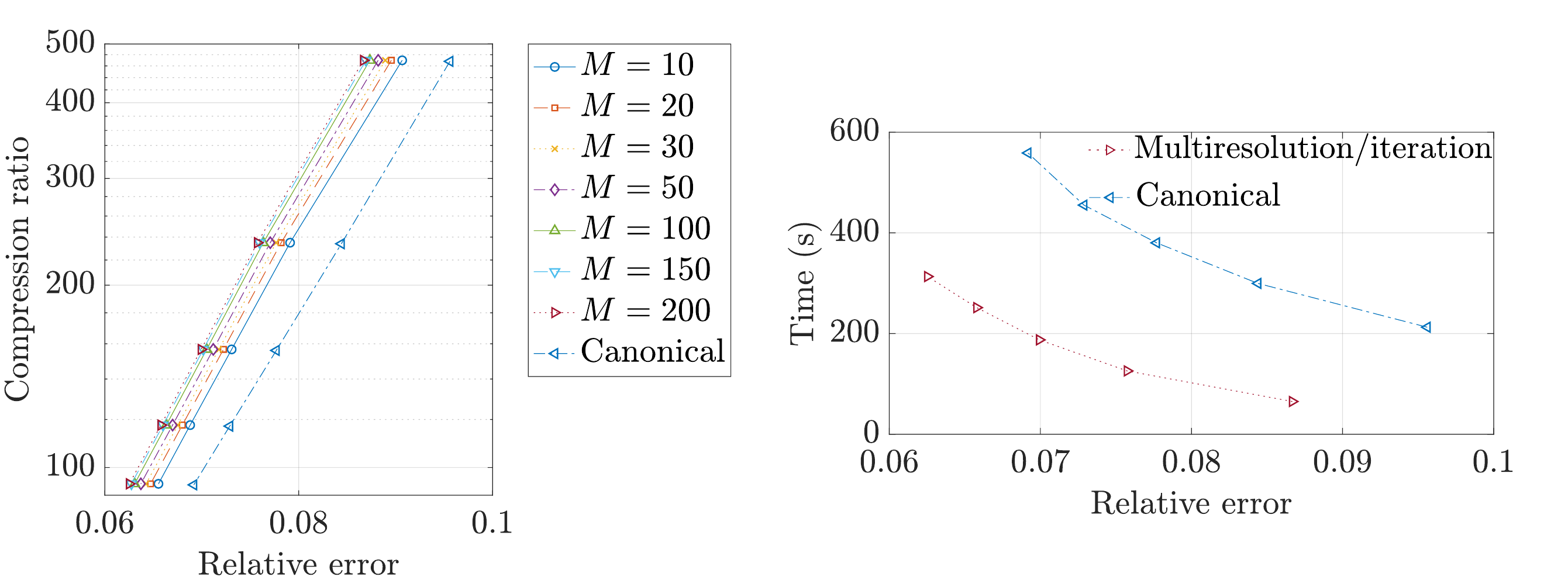}
\caption{Left: compression ratios of the tensor in \ref{sec:ex_video_candecomp} as function of approximation error and maximum number of iterations $M$. Right: runtimes as functions of approximation error.} \label{fig:ex_video_candecomp}
\end{center}
\end{figure}

\section{Conclusion}
We have studied a simple black-box tensor format for representing multidimensional data with multiple length-scales. An alternating algorithm for tensor approximation into this format was provided, and local convergence guarantees were proven. The closedness and stability properties of the format were also characterized. The efficiency of the format was numerically verified on six real-world datasets, achieving compression ratios several times higher than their counterparts for the tensor-train format, at the expense of higher run-times.

\appendix

\section{Proof of Thm.~\ref{thm:closed}}\label{appendix:proof_closed}
We prove (1) first. Let $k$ be the highest index of the rank-vector for which $r_k \neq 0$. Note that the multiresolution format with rank vector $(r_0,\ldots, r_{k-1}, b_s^{k}, 0, \ldots , 0)$ is precisely the set of matrices $\text{ext}_{L-k}(T)$ for $T$ any matrix in $\mathbb{R}^{b_s^k \times b_s^k}$. It therefore coincides with the multiresolution format with rank vector $(0,\ldots, 0, b_s^{k}, 0, \ldots , 0)$, so we assume that $\mathbf{r} \neq (0, 0, \ldots, 0, r_k, 0, \ldots , 0)$ and $r_k < b_s^k$. There is then some $r_{i} > 0$ with $i < k$. We will produce a matrix $T$ in $\overline{MS_{\text{TT}_\mathbf{r}}} \setminus MS_{\text{TT}_\mathbf{r}}$, and first introduce some auxiliary variables for the construction.

Denote by $v^{(n)}$ the vector $v^{(n)} = \begin{bmatrix}
\sqrt{n+1} &
\sqrt{n-1} &
\ldots &
\sqrt{n-1}
\end{bmatrix}^T,
$ and by $u$ the vector $u =  (\underbrace{1,1,\ldots , 1}_{b_s-1}, 1- \frac{b_s}{b_s-1}).$ Let $m = \floor*{\frac{r_k+1}{b_s}}$ and $w$ a vector of length $\abs{w}=r_k+1-mb_s$ defined by $w = (1, \ldots, 1, 1 - \frac{\abs{w}}{\abs{w}-1})$. The $b_s^k \times b_s^k$-matrix with all entries equal to $n$ is of the form $\text{ext}_{k-i}(S)$ with $S$ the $b_s^i \times b_s^i$-matrix with all entries equal to $n$. Viewing this matrix as a contribution from scale $i$, we study the matrix $\text{ext}_{L-k}(T^{(n)})$, where
\begin{align}
T^{(n)} = \begin{bmatrix}
n & \cdots & n \\
\vdots & \ddots & \vdots \\
n & \cdots & n
\end{bmatrix} - v^{(n)}\otimes v^{(n)} \! \! - \text{diag}(0,0,\underbrace{1,1,\ldots , 1}_{b_s-2}, \underbrace{u, \ldots , u}_{m-1}, w, \!\!\! \! \underbrace{0, \ldots , 0}_{b_s^k - (r_k+1) \geq 0} \!\!\! \!),
\end{align}
Since the second two terms have rank at most $r_k$, they are contained in the tensor format on the $k$:th scale. It follows that $\text{MS}_{\text{TT}_\mathbf{r}}$ contains $\text{ext}_{L-k}(T^{(n)})$. As $n \rightarrow \infty$, $n - \sqrt{n+1}\sqrt{n-1} = \frac{n^2 - (n^2-1)}{n + \sqrt{n^2-1}} \rightarrow 0$, so $T^{(n)}$ tends to the matrix $T$ defined by
\begin{equation}
T = \begin{bmatrix}
-1 & 0 & \cdots & 0 \\
0 & 1 & \cdots  & 1 \\
\vdots & \vdots & \ddots &\vdots  \\
0 & 1 & \cdots & 1
\end{bmatrix} - \text{diag}(0,0,\underbrace{1,1,\ldots , 1}_{b_s-2}, \underbrace{u, \ldots , u}_{m-1}, w, \underbrace{0, \ldots , 0}_{b_s^k - (r_k+1) \geq 0}).
\end{equation}
To conclude the proof, we show that $\text{ext}_{L-k}(T)$ is not contained in $MS_{\text{TT}_\mathbf{r}}$. Note that any sum of the form $\sum_{m = 0}^{k-1} \text{ext}_{k-m}(T_m)$ is necessarily constant on all batch-blocks of size $b_s$. By this, we mean the sub-matrices with row-indices and column-indices $i,j$ contained between two consecutive multiples of $b_s$. It therefore suffices to show that $\text{rank}(T+S) \geq r_k + 1$ for any matrix $S$ constant on all batch-blocks. We verify this by reducing $T+S$ to row-echelon form. 

We first decompose $T+S$ into its batch-blocks, i.e., write
\begin{equation}
T +S =  \begin{bmatrix}
M_{1,1} & M_{1,2} & \cdots & M_{1,b_s^{k-1}} \\
M_{2,1} & M_{2,2} & \cdots  & M_{2,b_s^{k-1}} \\
\vdots & \vdots & \ddots &\vdots  \\
M_{b_s^{k-1},1}& M_{b_s^{k-1},2} & \cdots & M_{b_s^{k-1},b_s^{k-1}},
\end{bmatrix} 
\end{equation}
with each $M_{i,j}$ of size $b_s\times b_s$. Here, each lower triangular batch-block $M_{i,j}$ for $i > j$ has constant rows, and each upper triangular batch-block $M_{i,j}$ for $i < j$ has constant columns. We will show that $T+S$ has rank at least $r_k+1$ by showing that the lower triangular batch-blocks can be reduced to $0$ by row operations, and by counting the pivots of the diagonal batch-blocks $M_{i,i}$. We will treat the case $i=1$ and $i>1$ separately, and start by showing that $M_{1,1}$ has full rank, by induction on $b_s$. For some constant $a$, $M_{1,1}$ can be written as
\begin{equation}
M_{1,1}=\begin{bmatrix}
a - 1 & a & a & \ldots & a  \\
a& a+1 & a+1 & \ldots & a+1  \\
a& a +1& a & \ldots & a+1 \\
\vdots & \vdots & & \ddots & \vdots  \\
 a & a+1 & a +1& \ldots & a 
 \end{bmatrix}.
\end{equation}
Subtracting the second row from the last and using the last row to eliminate the last column reduces $M_{1,1}$ to the block-form
\begin{equation}
\begin{bmatrix}
a - 1 & a & a & \ldots & a & 0 \\
a& a+1 & a+1 & \ldots & a+1 & 0 \\
a& a +1& a & \ldots & a+1 & 0\\
\vdots & \vdots & & \ddots & \vdots  \\
 a & a+1 & a +1& \ldots & a &0\\
0 & 0 & 0 & \ldots & 0& -1
 \end{bmatrix},
\end{equation}
which has full rank, by the induction hypothesis. The base case $b_s = 2$ corresponds to $M_{1,1} = \bigl[ \begin{smallmatrix} a-1 & a \\ a & a+1\end{smallmatrix}\bigr] $. This has full rank, which concludes the proof that $M_{1,1}$ has full rank. 

Moreover, since the batch-blocks $M_{1,2}, M_{1,3}, \ldots , M_{1,b_s^{k-1}}$ have constant columns, the row operations that reduce $M_{1,1}$ to row-echelon form preserve these constant columns. We can continue the row operations to reduce $M_{2,1}, M_{3,1} , \ldots M_{b_s^{k-1},1}$ to zero. Since each $M_{2,1}, M_{3,1} , \ldots M_{b_s^{k-1},1}$ has constant rows, after these operations, $M_{2,2}$ can still be written on the form

\begin{equation}
M_{2,2} = \begin{bmatrix}
c & c+1 & \ldots & c+1\\
c+1 & c & \ldots  &  \vdots \\
\vdots & \vdots & \ddots & c+1  \\
c+1& c+1  & \ldots & c + \frac{b_s}{b_s-1}
\end{bmatrix},
\end{equation}
for some real number $c$. Subtracting the last row from each preceding row and then subtracting from the last row $c+1$ times each preceding row results in

\begin{equation}
\begin{bmatrix}
-1 &  0&  \ldots &    \frac{-1}{b_s-1} \\
  0 &  -1&   \ldots  &   \frac{-1}{b_s-1} \\
 \vdots & \vdots &  \ddots & \vdots \\
0&  0& \ldots &   \frac{1}{b_s-1} 
\end{bmatrix},
\end{equation}
with an additional $b_s$ pivots. Again, the row operations used in the reduction of $M_{2,2}$ do not change the fact that $M_{2,3}, M_{2,4}, \ldots , M_{2, b_s^{k-1}}$ all have constant columns.

We iterate this argument for $i = 2, 3, \ldots, m+1$, and count the number of pivots. It follows that $S+T$ has rank no less than $b_s + (m-1)b_s + r_k +1 - mb_s  =  r_k +1$, which concludes the proof of the first assertion.

For the second statement, let $T$ and $T^{(n)}$ be as above and write $S = T\otimes e_1^{\otimes d - 2}$, $S^{(n)} = T^{(n)}\otimes e_1^{\otimes d - 2}$. Here, each $T^{(n)}$ is a matrix of rank at most $(r_k)_1$. Each $S^{(n)}$ can therefore be written in the TT-format with rank vector $(1,(r_k)_1,1\ldots , 1)$, so clearly $S^{(n)} \in MS_{TT_{\mathbf{r}}}$. Moreover, $S^{(n)} \rightarrow S$. If $\text{ext}_{L-k}(S)$ were in $MS_{TT_{\mathbf{r}}}$, i.e., $S$ were expressible in the form $S = \sum_{m=0}^{k} \text{ext}_{k-m}(S_m)$ with $\text{rank}_{TT}(S_m) \leq r_m$, then it would follow that
\begin{equation}
\begin{split}
T &= S\times_3 e_1 \times_4 \ldots \times_d e_1 = \sum_{m=0}^{k} \text{ext}_{k-m}(S_k)\times_3 e_1 \times_4 \ldots \times_d e_1 \\
&= \sum_{m=0}^{k} \text{ext}_{k-m}(S_k\times_3 e_1 \times_4 \ldots \times_d e_1), 
\end{split}
\end{equation}
so the matrix $T$ would be in $MS_{\mathbf{q}}$ with the rank vector $\mathbf{q} = ((r_0)_1, \ldots , (r_k)_1, 0, \ldots , 0)$. This contradicts the first statement and concludes the proof.

\section{Local convergence of Alg.~\ref{alg:AL-multi} with restructured sweeping order}\label{sec:local}
We will consider a modified version of Alg.~\ref{alg:AL-multi}, with differently structured sweeps. This makes it slower than Alg.~\ref{alg:AL-multi}, but both algorithms achieve similar compression ratios in our examples. 

Alg.~\ref{alg:AL-multi} improves each scale in every iteration. We now consider a modification that improves only one scale until convergence, and then moves on to the remaining scales successively. In detail, fix a maximum iteration number $M$. For each $k= 0, 1, \ldots , L$ in turn, the procedure computes an approximation $T_k^{(n)}$, for $n=1, \ldots , M$ on the $k$:th level, by calling Alg.~\ref{alg:AL-multi} on the tensor $T - \sum_{\ell = 0}^{k-1} \text{ext}_{L-\ell}(T_\ell^{(M)})$ with the two-level rank vector
\begin{equation}
(\underbrace{0, \ldots, 0}_{k-1}, r_k, \underbrace{0, \ldots , 0}_{L-k-1}, r_{k+1} + \ldots  + r_L).
\end{equation}
This produces a tensor $T_k^{(M)}$ on scale $k$ and a tensor $S^{(M)}_k$ on scale $L$. $T_k^{(M)}$ is stored and $S_k^{(M)}$ is discarded. Sweeping through all scales results in the multiresolution approximation $\sum_{\ell = 0}^{L} \text{ext}_{L-\ell}(T_\ell^{(M)}).$

We will consider tensors with a globally optimal approximation $\sum_{k = 0}^{L} \text{ext}_{L-k}(T_k)$ in the multiresolution format. Denote by $R$ the optimal approximation residual, i.e., $R = T - \sum_{k = 0}^{L} \text{ext}_{L-k}(T_k)$. While updating scales $k$ and $L$ during iteration $n$ of a run of the algorithm, the update equations in Lemma~\ref{lemma:updateS} read as
\begin{align*}
T_{k}^{(n+1)} &= \text{round}_{\mathcal{F}_{r_k}}\Bigl( \text{ave}_{L-k}\bigl(T - \sum_{\ell=0}^{k-1} \text{ext}_{L-\ell}(T_{\ell}^{(M)}) - S_k^{(n)}\bigr) \Bigr)  \\
&=  \text{round}_{\mathcal{F}_{r_k}}\Bigl( T_k + \text{ave}_{L-k}\bigl(  R - S_k^{(n)}\\
& \qquad \qquad \qquad + \sum_{\ell=k+1}^{L} \!\!\! \text{ext}_{L-\ell}(T_{\ell})  + \!\! \sum_{\ell=0}^{k-1}  \!\! \text{ext}_{L-\ell}(T_{\ell}- T^{(M)}_\ell) \bigr) \Bigr),\tag{\stepcounter{equation}\theequation} \\
S_k^{(n+1)} &= \text{round}_{\mathcal{F}_{r_{k+1} + \ldots + r_L}}\Bigl( T - \sum_{\ell=0}^{k-1} \text{ext}_{L-\ell}(T_{\ell}^{(M)}) - T_k^{(n+1)}\Bigr) \\
&= \text{round}_{\mathcal{F}_{r_{k+1} + \ldots + r_L}}\Bigl( \sum_{\ell=k+1}^{L} \!\!\! \text{ext}_{L-\ell}(T_{\ell}) \\
& \qquad \qquad \qquad  + R + \text{ext}_{L-k}(T_k- T_k^{(n+1)}) + \sum_{\ell=0}^{k-1}  \text{ext}_{L-\ell}(T_{\ell}- T^{(M)}_\ell ) \Bigr) .
\end{align*}

The key step of the convergence proof will be to analyze the following two amalgamated residual terms
\begin{equation}\label{eq:errors}
\begin{split}
E_{k}^{(n)} &:= \text{ave}_{L-k}\left( R - S^{(n)}_k + \sum_{\ell=k+1}^{L} \text{ext}_{L-\ell}(T_{\ell}) + \sum_{\ell=0}^{k-1}  \text{ext}_{L-\ell}(T_{\ell}- T^{(M)}_\ell)  \right), \\
D_{k}^{(n)} &:= R + \text{ext}_{L-k}(T_k- T_k^{(n+1)}) + \sum_{\ell=0}^{k-1}  \text{ext}_{L-\ell}(T_{\ell}- T^{(M)}_\ell ) .
\end{split}
\end{equation}
The idea of the proof is to show that our algorithm is locally a contraction, when the two residual terms $E_{k}^{(n)}$ and $D_{k}^{(n)}$ are sufficiently small. We will prove local convergence under the following three assumptions.
\begin{assumption}[existence of approximations]\label{ass:smooth}
The set of tensors $\mathcal{F}_{r}$ is a weakly closed, smooth manifold embedded in $\mathbb{R}^{b_s^L\times \ldots \times b_s^L}$ for any $r$.
\end{assumption}
\begin{assumption}[minimality]\label{ass:rank}
The tensors $\text{ext}_{L-k-1}(T_{k+1}) + \ldots + T_L$ have rank $r_{k+1} + \ldots + r_L$, for each $k = 0, \ldots, L$. 
\end{assumption}
\begin{assumption}[first-order analysis]\label{ass:sep}
There is an angle $\theta > 0$ such that, for any $n$ and $k$, the amalgamated errors $E_{k}^{(n)}$ and $D_{k}^{(n)}$ subtend an angle greater than $\theta$ to the tangent space of $\mathcal{F}_{r_k}$ at the point $T_k$, and of $\mathcal{F}_{r_{k+1} + \ldots + r_L}$ at the point $\text{ext}_{L-k-1}(T_{k+1}) + \ldots + T_L$, respectively.
\end{assumption}

The first assumption is required for optimal approximations in $\mathcal{F}_r$ to exist. The second assumption excludes non-minimal examples. For instance in the matrix case, if $r_0 = 1$, and $r_1 = b_s$, then $\text{ext}_{L}(T_0) + \text{ext}_{L-1}(T_1)$ also has rank $b_s \neq r_0 + r_1$. It would therefore be desirable for a properly designed algorithm to converge to a tensor in the multiresolution format with rank vector $(0, r_1, \ldots, r_L)$ instead of $(r_0, r_1, \ldots, r_L)$, to reduce storage cost. We therefore exclude these cases from consideration by imposing Assumption~\ref{ass:rank}. The third assumption is technical and made so that we can use first-order perturbation expansions as part of our analysis. It can be viewed as a non-degeneracy condition. Even though it is not verifiable a priori, should it not hold at iteration $n$, our algorithm is still nearly a contraction, with an error that can grow additively at most by a term proportional to $\|T - \sum_{k = 0}^{L} \text{ext}_{L-k}(T_k)\|^2$, as will be shown in Lemma~\ref{lemma:contract} below. Our main result is the following.

\begin{theorem}\label{thm:locconv}
Let $T$ be a tensor with an isolated, globally optimal approximation $\sum_{k=0}^L \text{ext}_{L-k} (T_k) $ in the multiresolution format with rank vector $(r_0, \ldots , r_k)$. Under assumptions \ref{ass:smooth} -- \ref{ass:sep}, there is then a constant $C > 0$ depending only on $T_0, \ldots , T_L$ such that
\begin{enumerate}
\item if $  T =  \sum_{k=0}^L \text{ext}_{L-k} (T_k)$ and $\| T_k - T_k^{(0)} \| \leq C$, then when using the algorithm of this section, $T^{(n)}_k \rightarrow T_k$ linearly as $n\rightarrow \infty$.
\item More generally, if
\begin{equation}
\begin{split}
\|T - \sum_{k=0}^L \text{ext}_{L-k} (T_k) \| \leq C, \\
\| T_k - T_k^{(0)} \| \leq C,
\end{split}
\end{equation}
then when using the algorithm of this section,
\begin{equation}\label{eq:locbound}
\| T_k - T_k^{(n)}\| \leq C_k \|T - \sum_{k=0}^L \text{ext}_{L-k} (T_k)\|
\end{equation}
for all $n$ large enough and some real constants $C_k$. In particular, there is a convergent subsequence $T^{(n_\ell)}_k$ such that $T^{(n_\ell)}_k \rightarrow S_k$ for all $k=0, \ldots , L$, where
\begin{equation}\label{eq:locboundconv}
\| T_k - S_k\| \leq C_k \|T - \sum_{k=0}^L \text{ext}_{L-k} (T_k)\|.
\end{equation}
\end{enumerate}
\end{theorem}

Assumption \ref{ass:smooth} is satisfied e.g., for the low-rank matrix format as well as for the TT-format \cite{holtz2012manifolds,steinlechner2016riemannian}. However, for the TT-format, note that the conclusion in Thm.~\ref{thm:locconv} applies when finding optimal low-rank approximations in the algorithm, rather than the quasi-optimal ones returned by the TT-SVD algorithm.

The result of Thm.~\ref{thm:locconv} is illustrated in Fig.~\ref{fig:ex_locconv}, for a matrix $T =\! \sum_{k=0}^L \text{ext}_{L-k} (T_k)$ and rank-vector $(0,0,0,8,0,10,10,10)$, which satisfy Assumption~\ref{ass:rank}. The matrices $T_k$ on each grid-scale were chosen with i.i.d. standard normal entries, and then normalized to have $\|T_k\| = 1$. The initial guess $\|T_k^{(0)}\|$ on each scale was chosen also with i.i.d. standard normal entries, with $\|T_k - T_k^{(0)}\| = 0.1$. We used $b_s = 2$ and $n = 128$.

\begin{figure}
\begin{center}
  \includegraphics[width=1.0\textwidth]{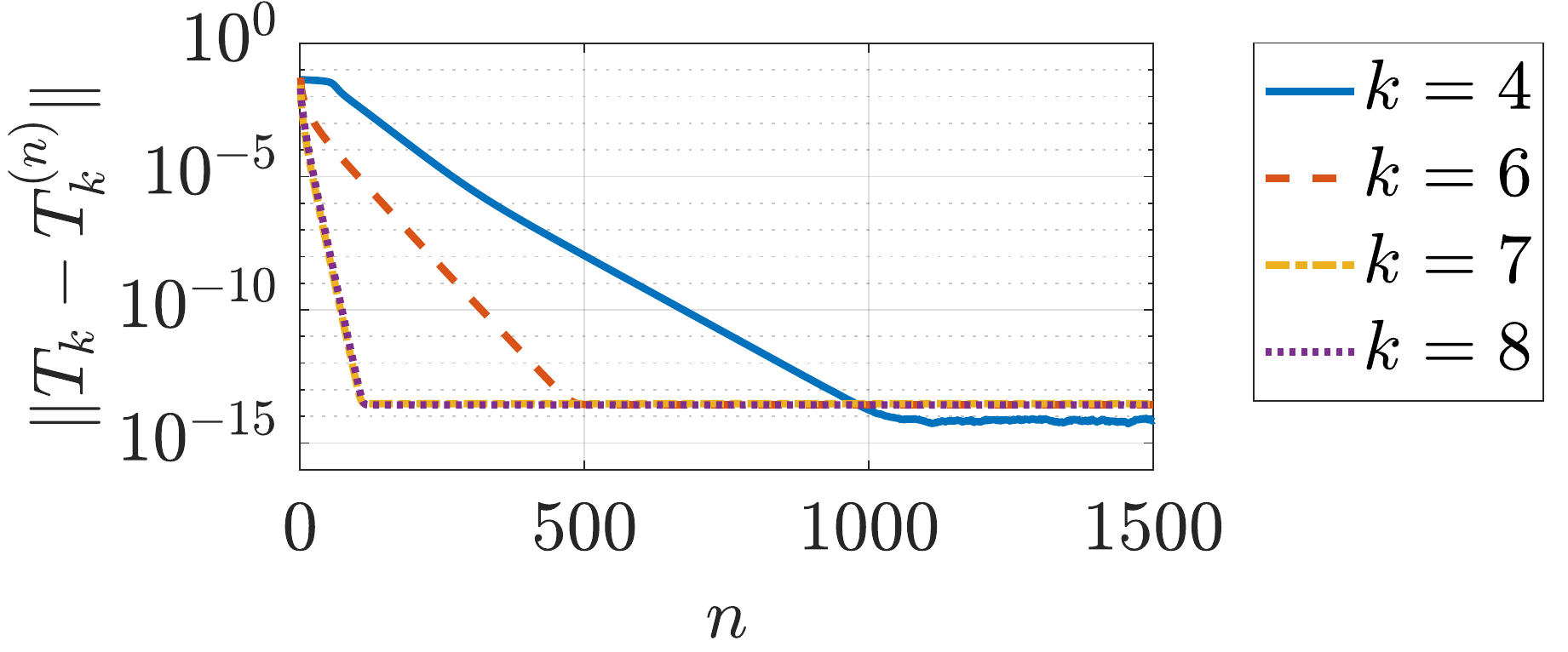}
  \caption{Example of local convergence guaranteed by Thm.~\ref{thm:locconv}.}
  \label{fig:ex_locconv}
\end{center}
\end{figure}

We will need the following Lemma.
\begin{lemma}\label{lemma:contract}
Let $x$ be a point on a smooth manifold $M$ embedded in $\mathbb{R}^N$, and $y$ a point in $\mathbb{R}^N$ s.t. the vector $y-x$ subtends an angle greater than $\theta$ with the tangent plane to $M$ at $x$. There is then a neighborhood of $x$ in $M$ and a constant $C<1$ which only depend on $x$ and $\theta$, such that the projection $\text{proj}_M$ onto the manifold, is uniquely defined in this neighborhood and satisfies
\begin{equation}
\| \text{proj}_M(y) - x \| \leq C \|y-x\|,
\end{equation}
\end{lemma}
for all $y$ in the neighborhood.
\begin{proof}
Throughout the proof, we will denote by $D$ arbitrary constants that only depend on $x$ and $\theta$. For ease of notation, we will not keep track of the exact expression for $D$, and the constants can be redefined even on the same line.

By the assumptions and possibly after permuting the coordinates, there is a smooth function $f$ and a neighborhood of radius $r$ around $x$ wherein points on $M$ can be written as $(z,f(z))$. Write $x = (x_0, f(x_0))$. Moreover, by the $\varepsilon$-neighborhood theorem \cite[p.~69]{guillemin2010differential}, there is a neighborhood of radius $\eta$ around $x$ such that the projection onto $M$ is uniquely defined. Take $\varepsilon = \frac{1}{2}\text{min}(r,\eta)$.

Since $\text{proj}_M(y)$ is the closest point on $M$ to $y$ and hence closer than $x$, $\|\text{proj}_M(y) - y\| \leq \|x - y\|$. Any $y$ in the $\varepsilon$-neighborhood of $x$ then has $\text{proj}_M(y)$ of the form $(z_0, f(z_0))$ with $\|x_0 - z_0\| \leq 2\varepsilon$, since $\|x_0 - z_0\| \leq \|\text{proj}_M(y) - x\| \leq \|\text{proj}_M(y) - y\|  + \|x - y\| \leq 2\|x-y\| \leq 2\varepsilon \leq r$.

Since $f$ is smooth, the shift of the tangent space $T_xM$ to $x$ consists of all points of the form $\bigl(z, f(x_0) + df(x_0)(z-x_0)\bigr)$. There is also a constant $D$ such that
\begin{equation}\label{eq:helper_manif}
\| f(z) - f(x_0) - df(x_0)(z-x_0) \| \leq D \|z-x_0\|^2,
\end{equation}
for $z$ in the $r$-neighborhood of $x_0$. Denote now by $y_p$ the orthogonal projection of $y$ onto $T_x$. Let also $\text{proj}_M(y) = \bigl(z_0, f(z_0)\bigr)$ be the projection of $y$ onto $M$. Since $f$ is smooth, $d f$ varies smoothly so it follows that 
\begin{equation}\label{eq:z2x}
\begin{split}
\| d f(z_0) - d f(x_0)\| &\leq D \|z_0 - x_0 \| \leq D \| \text{proj}_M(y) - x\| \\
&\leq D\left( \| \text{proj}_M(y) - y \| + \| y - x\| \right) \leq 2D\|x-y\|.
\end{split}
\end{equation}
It follows that the line passing through $y$ and $\text{proj}_M(y)$ is not contained in $T_xM$. If it were, we would be able to write $y -\text{proj}_M(y) = (\xi, df(x_0)\xi)$. Since $y -\text{proj}_M(y)$ is orthogonal to $T_{\text{proj}_M(y)}M$, this would imply
\begin{equation}
\begin{split}
0 &= \langle  (\xi, df(x_0)\xi),  (\xi, df(z_0)\xi) \rangle \\
&= \langle  (\xi, df(x_0)\xi),  (\xi, df(x_0)\xi) \rangle + \langle  df(x_0)\xi,  (df(z_0) - df(x_0))\xi \rangle \\
& \geq \| (\xi, df(x_0)\xi) \|^2 \cdot (1 - \| (df(z_0) - df(x_0))\|) > 0,
\end{split}
\end{equation}
by Eq.~\eqref{eq:z2x} for $\varepsilon$ sufficiently small, which is a contradiction

The line through $y$ and $\text{proj}_M(y)$ therefore intersects the shift of $T_xM$ to $x$ at some point $y_i$. Lastly, we denote the lift of $\text{proj}_M(y)$ onto the shift of $T_xM$ to $x$ by $y_\ell =  \bigl(z_{0}, f(x_0) + df(x_0)(z_{0}-x_0)\bigr)$. We have
\begin{equation}
\begin{split}\label{eq:linearquadratic}
\| \text{proj}_M(y) - x \| &\leq \| \text{proj}_M(y) - y_\ell \| + \|y_\ell - y_i\| + \|y_i - y_p\| + \|y_p - x\|,
\end{split}
\end{equation}
and we proceed by showing that the first three terms are bounded by constants times $\|y - x\|^2$.

For the first term, Eq.~\eqref{eq:helper_manif} gives
\begin{equation}
\| \text{proj}_M(y) - y_\ell \| = \| f(z_0) - f(x_0) - df(x_0)(z-x_0) \| \leq D \|z_0-x_0\|^2.
\end{equation}
Now, we have $\|z_0 - x_0\| \leq \|\text{proj}_M(y) - x \| \leq \|\text{proj}_M(y) - y \| + \| y - x\| \leq 2 \|y - x\|$. It follows that
\begin{equation}\label{eq:1stterm}
\| \text{proj}_M(y) - y_\ell \| \leq D \|y - x\|^2.
\end{equation}

We next study the third term. The vector $y-y_i$ is orthogonal to any vector in $T_{\text{proj}_M(y)}M$, which are of the form $(\xi, df(z_0)\xi)$. Likewise, $y-y_p$ is orthogonal to $T_xM$, i.e. vectors of the form $(\xi, df(x_0)\xi)$. If we write $y-y_i = \bigl( (y - y_i)_1, (y-y_i)_2 \bigr)$, this means that
\begin{equation}
\begin{split}
0 &= \langle y-y_i, (\xi, df(z_0)\xi) \rangle - \langle y-y_p, (\xi, df(x_0)\xi) \rangle = \\
&= \langle y_p-y_i, (\xi, df(x_0)\xi) \rangle + \langle (y-y_i)_2, (df(z_0) - df(x_0))\xi \rangle 
\end{split}
\end{equation}
Since $y_p - y_i$ is in $T_xM$, it can be written on the form $y_p - y_i = \bigl((y_p - y_i)_1, df(x_0)(y_p - y_i)_1\bigr)$, so
\begin{equation}
\begin{split}
- (y-y_i)_2^T (df(z_0) - df(x_0))\xi &= - \langle (y-y_i)_2, (df(z_0) - df(x_0))\xi \rangle  \\
&=\langle y_p-y_i, (\xi, df(x_0)\xi) \rangle \\
&=  (y_p-y_i)_1^T\left[ I + df(x_0)^Tdf(x_0)\right]\xi,
\end{split}
\end{equation}
for any vector $\xi$. This implies
\begin{equation}
(df(z_0) - df(x_0))^T (y-y_i)_2 =  \left[ I + df(x_0)^Tdf(x_0)\right](y_p-y_i)_1.
\end{equation}
Since $I$ is positive definite and $df(x_0)^Tdf(x_0)$ is positive semidefinite, the matrix in the right hand side is positive definite and hence invertible. This results in
\begin{equation}\label{eq:3dtermhelp}
\|(y_p-y_i)_1\| \leq  \bigl\| \left[ I + df(x_0)^Tdf(x_0)\right]^{-1}\bigr\| \cdot \|(df(z_0) - df(x_0))^T\| \cdot \|(y - y_i)_2\|.
\end{equation}
We observe that
\begin{align}
&\|y_p - y_i\| = \|\bigl((y_p - y_i)_1, df(x_0)(y_p - y_i)_1\bigr)\| \leq (1+\|df(x_0)\|)\|(y_p - y_i)_1\| \\
&\|(y - y_i)_2\| \leq \|y - y_i\| \leq \|y - y_p\| + \|y_p - y_i\| \leq \|y-x\| + \|y_p - y_i \| \\
&\|(df(z_0) - df(x_0))^T\| \leq  D\|z_0 - x_0\| \leq 2D\|y-x\|,
\end{align}
where the second equation used the fact that $y_p$ is the closest point to $y$ on the shift of $T_xM$ to $x$ which also includes $x$, and the third equation used Eq.~\eqref{eq:z2x}. Inserting this into Eq.~\eqref{eq:3dtermhelp} results in
\begin{equation}
\begin{split}
\|y_p - y_i\| &\leq 2D(1+\|df(x_0)\|)\bigl\| \left[ I + df(x_0)^Tdf(x_0)\right]^{-1}\bigr\|  \cdot \\
&\cdot \|y-x\|  (\|x-y\| + \|y_p - y_i \|) \leq D \|x-y\|^2 + \varepsilon D \|y_p - y_i\|.
\end{split}
\end{equation}
Moving the last term to the left hand side shows that
\begin{equation}\label{eq:3dterm}
\|y_p - y_i\| \leq D\|x-y\|^2,
\end{equation}
for $\varepsilon$ sufficiently small.

We lastly study the second term. Note that the vector $y_i - y_\ell$ is in $T_xM$ and can therefore be written as $y_i - y_\ell = \bigl( (y_i - y_\ell)_1, df(x_0)(y_i - y_\ell)_1\bigr)$. Next write $y_i - \text{proj}_M(y) = \bigl( (y_i - \text{proj}_M(y))_1, (y_i - \text{proj}_M(y))_2 \bigr)$. Since this vector is orthogonal to $T_{\text{proj}_M(y)}M$
\begin{equation}
\begin{split}
0 &= (y_i - \text{proj}_M(y))\cdot \left(\xi, df(z_0)\xi\right) \\
&= (y_i - \text{proj}_M(y))\cdot \left(\xi, df(x_0)\xi\right) + (y_i - \text{proj}_M(y))_2\cdot  \left(df(z_0) - df(x_0)\right)\xi,
\end{split}
\end{equation} 
for any vector $\xi$. Taking $\xi = (y_{i} - y_{\ell})_1$ results in
\begin{equation}
0 = (y_i - \text{proj}_M(y))\cdot (y_i - y_\ell) + (y_i - \text{proj}_M(y))_2\cdot  \left(df(z_0) - df(x_0)\right) (y_{i} - y_{\ell})_1,
\end{equation}
so
\begin{equation}
\begin{split}
\abs{ (y_i - \text{proj}_M(y)) \cdot (y_i - y_\ell) } &\leq  \| (y_i - \text{proj}_M(y))_2\|   \| df(z_0) - df(x_0)\|  \|(y_{i} - y_{\ell})_1\| \\
&\leq D\| y_i - \text{proj}_M(y)\|   \| x - y\|  \|y_{i} - y_{\ell}\| 
\end{split}
\end{equation}
This implies that the angle $\alpha$ between the vectors $y_i - \text{proj}_M(y)$ and $y_i - y_\ell$ satisfies $\cos \alpha \leq D\|x-y\| \leq D\varepsilon$. By choosing $\varepsilon$ sufficiently small, we can therefore guarantee $\sin \alpha \geq \frac{1}{2}$. The law of sines for the triangle with vertices $y_i$, $y_\ell$ and $\text{proj}_M(y)$ therefore gives
\begin{equation}\label{eq:2ndterm}
\| y_i - y_\ell\| \leq \frac{\|y_\ell - \text{proj}_M(y) \|}{\sin \alpha} \leq 2D \|x-y\|^2.
\end{equation}

Inserting Eqs.~\eqref{eq:1stterm},\eqref{eq:3dterm},\eqref{eq:2ndterm} into Eq.~\eqref{eq:linearquadratic} results in
\begin{equation}
\begin{split}
\| \text{proj}_M(y) - x \| &\leq \| y_p - x \| + D\|y-x\|^2.
\end{split}
\end{equation}
By assumption, the angle between $y_p - x$ is at least $\theta$, so
\begin{equation}
\|y_p - x \| \leq \|y-x\| \cos \theta.
\end{equation}
By choosing $\varepsilon$ so that $D\varepsilon + \cos \theta < 1$, it follows that
\begin{equation}
\begin{split}
\| \text{proj}_M(y) - x \| &\leq \| y_p - x \| + D\|y-x\|^2 \leq (\cos \theta + D\|x-y\|)\|x-y\| \\
&\leq (\cos \theta + D\varepsilon)\|x-y\|,
\end{split}
\end{equation}
which finishes the proof with $C = 1 - (\cos \theta + D\varepsilon)$, since $0 < C < 1$.
\end{proof}
We can now prove the main result of this section.
\begin{proof}[Proof of Thm.~\ref{thm:locconv}]
Write $R = T - \sum_{k=0}^L \text{ext}_{L-k} (T_k)$, and first assume that $\|R\| \neq 0$. We proceed by induction and assume that the algorithm has computed $T^{(M)}_0$, $\ldots$ , $T^{(M)}_{k-1}$ with $\| T_\ell- T_\ell^{(M)}\| \leq C_k \|R\|$, for $\ell = 1, \ldots , k-1$. When computing $T^{(n)}_k$, the algorithm proceeds by the updates 
\begin{equation}
\begin{split}
T_{k}^{(n+1)} &=  \text{round}_{\mathcal{F}_{r_k}}\Bigl( T_k + E_k^{(n)} \Bigr), \\
S_k^{(n+1)} &= \text{round}_{\mathcal{F}_{r_{k+1} + \ldots + r_L}}\Bigl( \sum_{\ell=k+1}^{L} \!\!\! \text{ext}_{L-\ell}(T_{\ell}) + D_k^{(n)} \Bigr),
\end{split}
\end{equation}
where $E_k^{(n)}$ and $D_k^{(n)}$ are defined in Eq.~\eqref{eq:errors}. This implies
\begin{equation*}
\begin{split}
&T_k^{(n+1)} - T_k = \text{round}_{\mathcal{F}_{r_k}}\left( T_k + E_k^{(n)}\right) - T_k, \\
& S_k^{(n+1)} -\!\!\!\! \sum_{\ell=k+1}^{L} \!\! \text{ext}_{L-\ell}(T_{\ell}) = \text{round}_{\mathcal{F}_{r_{k+1} + \ldots + r_L}}\bigl( \!\!\! \sum_{\ell=k+1}^{L} \!\!\! \text{ext}_{L-\ell}(T_{\ell}) + D_k^{(n)} \bigr) \!- \!\! \!\!\sum_{\ell=k+1}^{L} \!\! \text{ext}_{L-\ell}(T_{\ell}),
\end{split}
\end{equation*}
and we now show that the norm of these residual terms are bounded by the norms of $E_k^{(n)}$ and $D_k^{(n)}$, respectively. Provided the residual terms are sufficiently small, Lemma~\ref{lemma:contract} implies
\begin{equation}\label{eq:krate}
\begin{split}
\| &T_{k}^{(n+1)} - T_{k} \| \leq (1-C)\|E_k^{(n)}\| \\
&= (1-C)\| \text{ave}_{L-k}(\sum_{\ell=k+1}^{L} \text{ext}_{L-\ell}\bigl(T_{\ell}) - S_k^{(n)} + R + \sum_{\ell=0}^{k-1}  \text{ext}_{L-\ell}(T_{\ell}- T^{(M)}_\ell) \Bigr)\| \\
&\leq  \frac{(1-C)}{b_s^{\frac{d}{2}(L-k)}} \|\sum_{\ell=k+1}^{L} \text{ext}_{L-\ell}(T_{\ell}) - S_k^{(n)} + R + \sum_{\ell=0}^{k-1}  \text{ext}_{L-\ell}(T_{\ell}- T^{(M)}_\ell)\|,
\end{split}
\end{equation}
where the first step used Assumption~\ref{ass:sep} and the third step used Cauchy-Schwarz. Similarly, using Assumptions~\ref{ass:sep} and \ref{ass:rank}, we obtain
\begin{equation}\label{eq:convS}
\begin{split}
\| &S_k^{(n+1)} - \!\! \sum_{\ell=k+1}^{L} \!\! \text{ext}_{L-\ell}(T_{\ell}) \| \\
&\leq (1-C)\| \text{ext}_{L-k}(T_k- T_k^{(n+1)}) + R  + \sum_{\ell=0}^{k-1}  \text{ext}_{L-\ell}(T_{\ell}- T^{(M)}_\ell )\| \\
&\leq  (1-C)b_s^{\frac{d}{2}(L-k)}  \|T_k- T_k^{(n+1)}\| + (1-C_T)\|R\| + (1-C) \sum_{\ell=0}^{k-1} b^{\frac{d}{2}(L-\ell)}_s C_\ell \|R\| \\
&\leq (1-C)^2  \|S_k^{(n)} - \sum_{\ell=k+1}^{L} \text{ext}_{L-\ell}(T_{\ell}) - R + \sum_{\ell=0}^{k-1}  \text{ext}_{L-\ell}(T_{\ell}- T^{(M)}_\ell)\| \\
&\qquad + (1-C)\|R\| + (1-C) \sum_{\ell=0}^{k-1} b^{\frac{d}{2}(L-\ell)}_s C_\ell \|R\| \\
& \leq (1-C)^2  \|S_k^{(n)} - \sum_{\ell=k+1}^{L} \text{ext}_{L-\ell}(T_{\ell}) \| + D\|R\|,
\end{split}
\end{equation}
since $\|T_\ell - T_\ell^{(M)}\| \leq C_\ell \|R\|$, by the induction hypothesis. Here, $D$ is a constant dependent on $C$, $\{C_\ell\}_{\ell=0}^{k-1}$ and $b_s$. It follows that
\begin{equation*}
\begin{split}
\| S_k^{(n+1)} &\!\!\!- \!\! \!\! \sum_{\ell=k+1}^{L}\!\! \text{ext}_{L-\ell}(T_{\ell}) \| \leq (1-C)^{2n+2}  \|S_k^{(0)} - \!\!\!\! \sum_{\ell=k+1}^{L} \!\! \text{ext}_{L-\ell}(T_{\ell}) \| + D\|R\| \!\! \sum_{m=0}^{n+1} (1-C)^{2m} \\
&\leq (1-C)^{2n+2}  \|S_k^{(0)} - \sum_{\ell=k+1}^{L} \text{ext}_{L-\ell}(T_{\ell}) \| + \frac{D}{(1-C)^2}\|R\|.
\end{split}
\end{equation*}

Letting $n \rightarrow \infty$ shows that $\| S_k^{(n+1)} - \sum_{\ell=k+1}^{L} \text{ext}_{L-\ell}(T_{\ell}) \| \leq (1+\frac{D}{(1-C)^2})\|R\|$ for $n$ large enough. Inserting this into Eq.~\eqref{eq:krate} shows that 
\begin{equation}
\begin{split}
\| T_{k}^{(n+1)} - T_{k} \| &\leq  \frac{(1-C)}{b_s^{\frac{d}{2}(L-k)}} \left[ 2+\frac{D}{(1-C)^2}+ \sum_{\ell=0}^{k-1} b^{\frac{d}{2}(L-\ell)}_s C_\ell\right] \|R\|,
\end{split}
\end{equation}
which concludes the induction hypothesis and therefore the proof of the first statement. For the second statement, note that this implies that the $T_k^{(n)}$ are bounded, so there is a convergent subsequence by the Bolzano-Weierstrass theorem. Eq.~\eqref{eq:locboundconv} then follows from taking the limit of Eq.~\eqref{eq:locbound}.

Lastly, assume that $\|R\| = 0$. Eqs.~\eqref{eq:convS} and \eqref{eq:krate} imply that 
\begin{equation}
\begin{split}
\| T_{k}^{(n+1)} - T_{k} \| \leq (1-C)^{2n+3}  \|S_k^{(0)} - \sum_{\ell=k+1}^{L} \text{ext}_{L-\ell}(T_{\ell}) \|,
\end{split}
\end{equation}
which concludes the proof.
\end{proof}

\pagebreak
\bibliographystyle{siamplain}
\bibliography{references}

\end{document}